\documentclass{m2an}
\usepackage{enumerate}
\usepackage{subfigure}
\usepackage{graphicx}
\usepackage{float}
\newtheorem{theorem}{Theorem}[section]
\newtheorem{lemma}[theorem]{Lemma}

\theoremstyle{definition}

\theoremstyle{remark}
\newtheorem{remark}[theorem]{Remark}

\numberwithin{equation}{section}
\begin{document}
\title{Fast algorithms for convolution quadrature of Riemann-Liouville fractional derivative}
\author{Jing Sun}\address{School of Mathematics and Statistics, Gansu Key Laboratory of Applied Mathematics and Complex Systems, Lanzhou University, Lanzhou 730000, P.R. China}
\author{Daxin Nie}\address{School of Mathematics and Statistics, Gansu Key Laboratory of Applied Mathematics and Complex Systems, Lanzhou University, Lanzhou 730000, P.R. China}
\author{Weihua Deng}\address{School of Mathematics and Statistics, Gansu Key Laboratory of Applied Mathematics and Complex Systems, Lanzhou University, Lanzhou 730000, P.R. China}\email{dengwh@lzu.edu.cn}
\date{}
\begin{abstract}
	Recently, the numerical schemes of the Fokker-Planck equations describing anomalous diffusion with two internal states have been proposed in [Nie, Sun and Deng, arXiv: 1811.04723], which use convolution quadrature to approximate the Riemann-Liouville fractional derivative; and the schemes need huge storage and computational cost because of the non-locality of fractional derivative and the large scale of the system. This paper first provides the fast  algorithms for computing the Riemann-Liouville derivative based on convolution quadrature with the generating function given by the backward Euler and second-order backward difference methods; the algorithms  don't require the assumption of the regularity of the solution in time, while the computation time and the total memory requirement are greatly reduced. Then we apply the fast algorithms to solve the homogeneous fractional Fokker-Planck equations with two internal states for nonsmooth data and get the first- and second-order accuracy in time. Lastly, numerical examples are presented to verify the convergence and the effectiveness of the fast algorithms.
 \end{abstract}
\subjclass{26A33, 44A35, 65M06}
\keywords{convolution quadrature, fast algorithms, Riemann-Liouville derivative, fractional Fokker-Planck equations,  error estimates}
\maketitle
\section{Introduction}
Nowadays, it is widely recognized that anomalous diffusions are ubiquitous in the natural world, which also naturally become an interdisciplinary research topic. Microscopically, various stochastic processes are introduced, including L\'{e}vy process, L\'{e}vy walk, L\'{e}vy flights, continuous time random walks with power law waiting times and/or jump lengths, while macroscopically the diverse partial differential equations (PDEs) governing the probability density functions (PDFs) of a variety of statistical observables, e.g., position and functionals, are derived \cite{Chen2018}. A lot of efforts are made for numerically solving these PDEs \cite{Alikhanov2015,Chen2009,Deng2009,Deng2013,Li2009,Li2017_1,Lin2007}, and generally they are nonlocal, which urges people to develop the fast algorithm to deal with the challenges of huge storage and computational complexity \cite{Jiang2017,Yan2017,Zeng2018}.

More recently, the anomalous diffusions with multiple internal states are carefully studied, and the corresponding macroscopic PDEs are built \cite{Xu2018,Xu20182}.
Then, \cite{Nie2018} provides a numerical scheme and does the numerical analyses for the  homogeneous fractional Fokker-Planck equations with two internal states \cite{Xu2018}, i.e.,
\begin{equation}\label{equmatrixequ}
	\left \{
	\begin{split}
		&\mathbf{M}^T\frac{\partial }{\partial t} \mathbf{G}=(\mathbf{M}^T-\mathbf{I}){\rm diag}(~_0D^{1-\alpha_1}_t, ~_0D^{1-\alpha_2}_t)\mathbf{G}\\
		&\quad\quad\quad\quad\quad\quad +\mathbf{M}^T{\rm diag}(~_0D^{1-\alpha_1}_t, ~_0D^{1-\alpha_2}_t)\Delta\mathbf{G}~~~~\quad\quad\quad \quad \ {\rm in}\  \Omega,\ t\in[0,T],\\
		&\mathbf{G}(\cdot,0)=\mathbf{G}_0 \quad\quad\quad\quad\quad\quad\quad\quad\quad\quad\quad\quad\quad\quad\quad\quad\quad\quad\quad\quad {\rm in}\ \Omega,\\
		&\mathbf{G}=0 \quad\quad\quad\quad\quad\quad\quad\quad\quad\quad\quad\quad\quad\quad\quad\quad\quad\quad\quad\quad\quad\quad\ \ \, {\rm on}\ \partial\Omega,\ t\in[0,T],
	\end{split}
	\right .
\end{equation}
where $\Omega$ denotes a bounded convex polygonal domain in $\mathbb{R}^d$ $(d=1,2,3)$; $\mathbf{M}$ is the transition matrix of a Markov chain, being a $2\times 2$ invertible matrix here, namely, we set
\begin{equation}\label{equmarM}
\mathbf{M}=\left[
\begin{matrix}
m& 1-m\\
1-m & m
\end{matrix}\right],\ m\in(0,1/2)\bigcup (1/2,1);
\end{equation}
 $\mathbf{G}=[G_1,G_2]^T$ denotes the solution of the system \eqref{equmatrixequ} and $\mathbf{G}_0=[G_1(0),G_2(0)]^T$ is the initial value; $\mathbf{I}$ is an identity matrix; `diag' denotes a diagonal matrix formed from its vector argument,  and $~_0D^{1-\alpha_{i}}_t$, $i=1,2$ are the Riemann-Liouville fractional derivatives defined by \cite{Podlubny1999}
\begin{equation}\label{equdefRie}
	_{0}D^{1-\alpha_{i}}_tG=\frac{1}{\Gamma(\alpha_{i})}\frac{\partial}{\partial t}\int^t_{0}(t-\xi)^{\alpha_{i}-1}G(\xi)d\xi, ~\alpha_{i}\in(0,1), ~i=1,2.
\end{equation}
The convolution quadrature \cite{Lubich1988,Lubich19882,Lubich1996} is a popular strategy to approximate \eqref{equdefRie}, since it doesn't require the assumption of regularity of the solution in time and it can achieve high order accuracy after the suitable modification \cite{Jin2003,Jin2014,Jin2015,Jin2016,Jin2017}. The backward Euler (BE) convolution quadrature is used to solve the system \eqref{equmatrixequ} with first-order accuracy in \cite{Nie2018}; it's worth pointing out that the time step size $\tau$ must be very small to ensure the stability and convergence of the algorithm when $m$ is less than but close to $0.5$, which results in huge storage cost and computational complexity. So fast algorithm is expected to be developed.


In this paper, besides the first-order approximation given in \cite{Nie2018}, we also discuss the second-order approximation of \eqref{equdefRie} designed by convolution quadrature with generating function given by second-order backward difference (SBD).
We modify the k-th order $(k=1,2)$ approximations of $~_0D^{\alpha}_{t}G(t_n)$ based on convolution quadrature to speed up the calculations, that is, we use the sum of geometric sequences to approximate the weights generated by BE and SBD convolution quadrature.  According to the property of the geometric sequences, the computation can be performed iteratively, which greatly reduce the computational complexity and the storage cost (for the details, refer to Section \ref{sec:3}).
Afterwards, we apply the designed fast algorithms to solve the system \eqref{equmatrixequ} and get the first- and second-order accuracy in time. Compared with the existing fast algorithms for fractional derivatives, the ones provided in this paper have the advantage of weakening the requirement of the regularity of the solution in time.


This paper is organized as follows. In Section \ref{sec:2}, we  introduce some needed notations and  lemmas. In Section \ref{sec:3}, we develop the fast algorithms based on convolution quadrature for Riemann-Liouville fractional derivatives, i.e., fast BE and SBD discretizations. In Section \ref{sec:4}, we use the fast algorithms to solve the homogeneous fractional Fokker-Planck equations \eqref{equmatrixequ} with the first- and second-order accuracy, respectively, in time. Section \ref{sec:5} shows the effectiveness of the fast algorithms by numerical experiments.

\section{Preliminaries} \label{sec:2}
Let's begin with some needed notations. Throughout the paper, $C$ denotes a generic positive constant, whose value may differ at each occurrence.
We denote $G_1(t)$, $G_2(t)$ as the functions $G_1(\cdot,t)$, $G_2(\cdot,t)$ respectively, introduce $\|\cdot\|$ as the operator norm from $L^2(\Omega)$ to $L^2(\Omega)$, and denote $A=-\Delta$ in the following. For any $ r\geq  0 $, denote the space $ \dot{H}^r(\Omega)=\{\vartheta\in L^2(\Omega): A^{\frac{r}{2}}\vartheta\in L^2(\Omega) \}$ with the norm \cite{Bazhlekova2015}
\begin{equation*}
	\|\vartheta\|^2_{\dot{H}^r(\Omega)}=\sum_{j=1}^{\infty}\lambda_j^r(\vartheta,\varphi_j)^2,
\end{equation*}
where $ {(\lambda_j,\varphi_j)} $ are the eigenvalues ordered non-decreasingly and the corresponding eigenfunctions normalized in the $ L^2(\Omega) $ norm of $-\Delta$ on the domain $\Omega$ with a zero Dirichlet boundary condition. Thus $ \dot{H}^0(\Omega)=L^2(\Omega) $, $\dot{H}^1(\Omega)=H^1_0(\Omega)$, and $\dot{H}^2(\Omega)=H^2(\Omega)\bigcap H^1_0(\Omega)$.

After that, for $\kappa>0$ and $\pi/2<\theta<\pi$, we define sectors $\Sigma_{\theta}$ and $\Sigma_{\theta,\kappa}$ in the complex plane $\mathbb{C}$ as
\begin{equation*}
	\begin{aligned}
		&\Sigma_{\theta}=\{z\in\mathbb{C}\setminus \{0\},|\arg z|\leq \theta\}, \\
		&\Sigma_{\theta,\kappa}=\{z\in\mathbb{C}:|z|>\kappa,|\arg z|\leq \theta\},\\
	\end{aligned}
\end{equation*}
and the contour $\Gamma_{\theta,\kappa}$ is defined by
\begin{equation*}
	\Gamma_{\theta,\kappa}=\{z\in\mathbb{C}: |z|=\kappa,|\arg z|\leq \theta\}\bigcup\{z\in\mathbb{C}: z=r e^{\pm \mathbf{i}\theta}: r\geq \kappa\},
\end{equation*}
oriented with an increasing imaginary part, where $\mathbf{i}$ denotes the imaginary unit and $\mathbf{i}^2=-1$.
According to the results in \cite{Nie2018}, the system \eqref{equmatrixequ} can be rewritten as
\begin{equation}\label{equrqtosol}
	\left \{
	\begin{aligned}
		&\frac{\partial G_1}{\partial t}+a~_0D^{1-\alpha_1}_tG_1+~_0D^{1-\alpha_1}_tA G_1=a~_0D^{1-\alpha_2}_tG_2\quad\quad\quad\,\, {\rm in}\ \Omega,\ t\in[0,T],\\
		&\frac{\partial G_2}{\partial t}+a~_0D^{1-\alpha_2}_tG_2+~_0D^{1-\alpha_2}_tA G_2=a~_0D^{1-\alpha_1}_tG_1 \quad\quad\,\,\quad {\rm in}\ \Omega,\ t\in[0,T],\\
		&\mathbf{G}(\cdot,0)=\mathbf{G}_0\quad \quad\quad\quad\quad\quad\quad\quad\quad\quad\quad\,\quad\quad\quad\quad\quad\quad\quad\quad\quad {\rm in}\ \Omega,\\
		&\mathbf{G}=0 \quad\quad\quad\quad\quad\quad\quad\quad\quad\quad\quad\quad\quad\quad\quad\quad\quad\quad\quad\quad\quad\quad\quad {\rm on}\ \partial\Omega,\ t\in[0,T],
	\end{aligned}
	\right .
\end{equation}
where  $ a=\frac{1-m}{2m-1} $ and $m$ is defined in \eqref{equmarM};
and the system \eqref{equrqtosol} has the solution of the form
\begin{equation*}
	\begin{aligned}
		\tilde{G}_1=&H_{\alpha_1}(z)z^{\alpha_1-1}G_1(0)+aH(z)z^{\alpha_1-1}G_2(0),\\
		\tilde{G}_2=&aH(z)z^{\alpha_2-1}G_1(0)+H_{\alpha_2}(z)z^{\alpha_2-1}G_2(0),
	\end{aligned}
\end{equation*}
 where `$\tilde{~}$' stands for taking Laplace transform,
\begin{equation}\label{equdefHz}
	H(z)=\left((z^{\alpha_1}+a+A)(z^{\alpha_2}+a+A)-a^2\right)^{-1},
\end{equation}
and
\begin{equation}\label{H12}
	\begin{aligned}
		H_{\alpha_1}(z)=H(z)(z^{\alpha_2}+a+A),\quad
		H_{\alpha_2}(z)=H(z)(z^{\alpha_1}+a+A).
	\end{aligned}
\end{equation}

Then we provide some estimates related to \eqref{equdefHz} and \eqref{H12}, which will be used in the error estimates.
\begin{lemma}[\cite{Nie2018}]\label{lemtheestimateofHblabla}
	When $z\in\Sigma_{\theta,\kappa}$, $ \pi/2<\theta<\pi $, and $\kappa>\max\left (2|a|^{1/\alpha_1},2|a|^{1/\alpha_2}\right )$, we have the estimates
	\begin{equation*}
		\begin{aligned}
			&\left \|\left(z^{\alpha_1}+a+A\right)^{-1}\right \|\leq C|z|^{-\alpha_1}, ~\left \|\left(z^{\alpha_2}+a+A\right)^{-1}\right \|\leq C|z|^{-\alpha_2},\\
			&\|H(z)\|\leq C|z|^{-\alpha_1-\alpha_2},\quad\|H_{\alpha_1}(z)\|\leq C|z|^{-\alpha_1},\quad\|H_{\alpha_2}(z)\|\leq C|z|^{-\alpha_2},\\
			&\|AH(z)\|\leq C\min\left (|z|^{-\alpha_1},|z|^{-\alpha_2}\right ),\quad\|AH_{\alpha_1}(z)\|\leq C,\quad\|AH_{\alpha_2}(z)\|\leq C,
		\end{aligned}
	\end{equation*}
where $H(z)$, $H_{\alpha_1}(z)$ and $H_{\alpha_2}(z)$ are defined in \eqref{equdefHz} and \eqref{H12}, respectively.
\end{lemma}
\begin{lemma}\label{lemaAH}
	When $z\in\Sigma_{\theta,\kappa}$, $ \pi/2<\theta<\pi $, and $\kappa>\max\left (2|a|^{1/\alpha_1},2|a|^{1/\alpha_2}\right )$, there are the estimates of $H(z)$, $H_{\alpha_1}(z)$ and $H_{\alpha_2}(z)$  in \eqref{equdefHz} and \eqref{H12},
	\begin{equation*}
		\begin{aligned}
			&\|(a+A)H(z)\|\leq C\min\left (|z|^{-\alpha_1},|z|^{-\alpha_2}\right ), ~\|(a+A)H_{\alpha_1}(z)\|\leq C, ~\|(a+A)H_{\alpha_2}(z)\|\leq C.\\
		\end{aligned}
	\end{equation*}
\end{lemma}
\begin{proof}
	First, consider the estimate of $\|(a+A)H(z)\|$. Obviously, there exist the equalities
	\begin{equation*}
		\begin{aligned}
			(a+A)H(z)&=(z^{\alpha_1}+a+A)H(z)-z^{\alpha_1}H(z)\\
			&=(z^{\alpha_2}+a+A)H(z)-z^{\alpha_2}H(z).
		\end{aligned}
	\end{equation*}
	To estimate $(a+A)H(z)$, one can estimate $(z^{\alpha_1}+a+A)H(z)$, $z^{\alpha_1}H(z)$, $(z^{\alpha_2}+a+A)H(z)$, and $z^{\alpha_2}H(z)$. As for $(z^{\alpha_1}+a+A)H(z)$, let
	\begin{equation*}
		(z^{\alpha_2}+a+A)u-a^2(z^{\alpha_1}+a+A)^{-1}u=v,
	\end{equation*}
	which results in
	\begin{equation*}
		u={\left(z^{\alpha_2}+a+A\right)^{-1}}v+a^2{\left((z^{\alpha_1}+a+A)(z^{\alpha_2}+a+A)\right)^{-1}}u.
	\end{equation*}
	Performing $L_2$ norm on both sides of the above equality and using Lemma \ref{lemtheestimateofHblabla}, we have
	\begin{equation*}
		\|u\|_{L^2(\Omega)}\leq C|z|^{-\alpha_2}\|v\|_{L^2(\Omega)}+Ca^2|z|^{-\alpha_1-\alpha_2}\|u\|_{L^2(\Omega)}.
	\end{equation*}
	Taking $\kappa$ sufficiently large leads to
	\begin{equation*}
		\|(z^{\alpha_1}+a+A)H(z)\|\leq C|z|^{-\alpha_2}.
	\end{equation*}
	Similarly, we also have
	\begin{equation*}
		\begin{aligned}
			&\|(z^{\alpha_2}+a+A)H(z)\|\leq C|z|^{-\alpha_1}.	
		\end{aligned}
	\end{equation*}
	From Lemma \ref{lemtheestimateofHblabla}, there exist
	\begin{equation*}
		\begin{aligned}
			&\|z^{\alpha_1}H(z)\|\leq C|z|^{-\alpha_2}, ~~~~~\|z^{\alpha_2}H(z)\|\leq C|z|^{-\alpha_1}.
		\end{aligned}
	\end{equation*}
	Thus, we obtain the estimate of $\|(a+A)H(z)\|$. The estimates of $\|(a+A)H_{\alpha_1}(z)\|$ and $\|(a+A)H_{\alpha_2}(z)\|$ can be similarly obtained.
\end{proof}

\section{Fast evaluation of the Riemann-Liouville fractional derivative} \label{sec:3}
In this section, we provide the fast BE and fast SBD approximations based on the convolution quadrature of the Riemann-Liouville fractional derivative. Suppose that $N$ is the total number of time steps, and the time step size $\tau=T/N$ and $t_n=n\tau$, $1\leq n\leq N$.

Let's start from the integral representation of the power function.
\begin{lemma}[\cite{Jiang2017}]\label{lemtb}
	For any $\beta>0$, there is
	\begin{equation*}
		\frac{1}{t^\beta}=\frac{1}{\Gamma(\beta)}\int_0^{\infty} e^{-ts}s^{\beta-1}ds.
	\end{equation*}
\end{lemma}

By using the property of convolution and Lemma $\ref{lemtb}$, Eq. \eqref{equdefRie} can be rewritten as
\begin{equation}\label{equcapdecompose}
	\begin{aligned}
		\!_0D^{\alpha}_t G(t)=&\frac{1}{\Gamma(1-\alpha)}\frac{\partial}{\partial t}\int_0^t\frac{G(\xi)}{(t-\xi)^{\alpha}}d\xi
		=\frac{1}{\Gamma(1-\alpha)}\frac{\partial}{\partial t}\int_0^t\frac{G(t-\xi)}{\xi^{\alpha}}d\xi\\
		=&\frac{1}{\Gamma(1-\alpha)\Gamma(\alpha)}\frac{\partial}{\partial t}\int_0^tG(t-\xi)\int_0^{\infty}e^{-\xi s}s^{\alpha-1}dsd\xi
		=\frac{1}{\Gamma(1-\alpha)\Gamma(\alpha)}\int_0^{\infty}s^{\alpha-1}\frac{\partial}{\partial t}\int_0^tG(t-\xi)e^{-\xi s}d\xi ds.
	\end{aligned}
\end{equation}
Taking the Laplace transform on the both sides of \eqref{equcapdecompose}, we obtain
\begin{equation}\label{zzzz}
	z^{\alpha}=\frac{1}{\Gamma(1-\alpha)\Gamma(\alpha)}\int_0^{\infty}s^{\alpha-1} \frac{z}{s+z}ds.
\end{equation}
Following the classical convolution quadrature, we only need to take $z=\delta(\zeta)$ on the left side of \eqref{zzzz} to get the discretization of $~_0D^\alpha_tG(t)$, where $\delta(\zeta)$ is the generating function given by BE or SBD methods, i.e., $\delta(\zeta)=(1-\zeta)/\tau$ or $\delta(\zeta)=((1-\zeta)+(1-\zeta)^2/2)/\tau$. Here we get the integral representation of the weights generated by convolution quadrature according to \eqref{zzzz} to speed up the evaluation.

\subsection{Fast BE discretization}
Firstly we  take  $z=\delta(\zeta)=(1-\zeta)/\tau$  in \eqref{zzzz} and get
\begin{equation}\label{equO1weightint}
	\left (\frac{1-\zeta}{\tau}\right )^{\alpha}=\frac{1}{\Gamma(1-\alpha)\Gamma(\alpha)}\int_0^{\infty}s^{\alpha-1} \left (1-\frac{\tau s}{\tau s+1-\zeta}\right )ds.
\end{equation}
Setting
\begin{equation}\label{equdefofdj1}
	\left (\frac{1-\zeta}{\tau}\right )^{\alpha}=\sum_{i=0}^{\infty}d^{\alpha}_{1,i}\zeta^i
\end{equation}
and using the fact
\begin{equation*}
	\frac{\tau s}{\tau s+1-\zeta}=\frac{\tau s}{1+\tau s}\sum_{i=0}^{\infty}\left(\frac{\zeta}{1+\tau s}\right)^i,
\end{equation*}
we obtain from \eqref{equO1weightint}
\begin{equation*}
		d^{\alpha}_{1,0}=\frac{1}{\Gamma(1-\alpha)\Gamma(\alpha)}\int_0^{\infty}s^{\alpha-1} \left (1-\frac{\tau s}{\tau s+1}\right )ds,
\end{equation*}
\begin{equation}\label{equO1intrep}
		d^{\alpha}_{1,i}=-\frac{1}{\Gamma(1-\alpha)\Gamma(\alpha)}\int_0^{\infty}s^{\alpha-1} \frac{\tau s}{(\tau s+1)^{i+1}}ds,\qquad i\geq 1.
\end{equation}
For $i\geq 2$, simple calculation (see Appendix \ref{AP1}) leads to
\begin{equation}\label{equO1toderinAp}
	\begin{aligned}
		d^{\alpha}_{1,i}=-\frac{1}{4\tau^{\alpha}\Gamma(1-\alpha)\Gamma(\alpha)}\int_{-1}^{1}(1-s)^{\alpha}(1+s)^{1-\alpha}\bar{d}_{1,i}(s)ds,
	\end{aligned}
\end{equation}
where
\begin{equation}\label{equdefofbardj1}
	\bar{d}_{1,i}(s)=\left(\frac{s+1}{2}\right)^{i-2}.
\end{equation}
Then the Riemann-Liouville fractional derivative can be discretized as
\begin{equation}\label{equrelationO1}
	\begin{aligned}
		\!_0D^{\alpha}_t G(t_n)&\approx\sum_{i=0}^{n}d^{\alpha}_{1,i}G(t_{n-i})\quad {\rm (classical~BE~discretization)}\\
		&=d^{\alpha}_{1,0} G(t_n)+d^{\alpha}_{1,1}G(t_{n-1})+\sum_{i=0}^{n-2}\sum_{j=1}^{N_p}w^\alpha_j\bar{d}_{1,n-i}(s^{\alpha}_j)G(t_i)+\sum_{i=0}^{n-2}\epsilon^{\alpha}_{1,n-i}G(t_i),
	\end{aligned}
\end{equation}
where
\begin{equation*}
	-\frac{1}{4\tau^{\alpha}\Gamma(1-\alpha)\Gamma(\alpha)}\int_{-1}^{1}(1-s)^{\alpha}(1+s)^{1-\alpha}\bar{d}_{1,i}(s)ds=\sum_{j=1}^{N_p}w^{\alpha}_{j}\bar{d}_{1,i}(s^{\alpha}_j)+\epsilon^{\alpha}_{1,i}, \quad i=2,3,\cdots,
\end{equation*}
$\{w^{\alpha}_j\}_{j=1}^{N_p}$ denote the integration weights, $\{s^{\alpha}_j\}_{j=1}^{N_p}$ signify the integration points, $N_p$ is the number of integration points, and $\{\epsilon^{\alpha}_{1,i}\}_{i=2}^{\infty}$ indicate the errors caused by integral approximation. Obviously, $\epsilon^{\alpha}_{1,0}=\epsilon^{\alpha}_{1,1}=0$. To make $|\epsilon^{\alpha}_{1,i}|$ small enough, we can use the Gauss-Jacobi rule to generate $w^\alpha_j$ and $s^{\alpha}_j$, and for $2<i\leq 2N_p+1$, $d^\alpha_{1,i}$ can be exactly approximated, i.e., $\epsilon^\alpha_{1,i}=0$.

To get a fast evaluation for \eqref{equrelationO1}, we rewrite it as
\begin{equation}\label{equO1fastRL}
	\begin{aligned}
		\!_0D^{\alpha}_t G(t_n)&\approx d^{\alpha}_{1,0} G(t_n)+d^{\alpha}_{1,1}G(t_{n-1})+\sum_{i=0}^{n-2}\sum_{j=1}^{N_p}w^{\alpha}_j\bar{d}_{1,n-i}(s^{\alpha}_j)G(t_i)\\
		&= d^{\alpha}_{1,0} G(t_n)+d^{\alpha}_{1,1}G(t_{n-1})+\sum_{j=1}^{N_p}\mathcal{G}^1_{hist,j}(t_n), \qquad n=2,3,\cdots,
	\end{aligned}
\end{equation}
where
\begin{equation}\label{equdefO1historyterm}
	\mathcal{G}^1_{hist,j}(t_n)=w^\alpha_j\sum_{i=0}^{n-2}\bar{d}_{1,n-i}(s^{\alpha}_j)G(t_i),
\end{equation}
and we call it history part. According to Eq. \eqref{equdefofbardj1}, it's easy to know that $\{\bar{d}_{1,i}(s)\}_{i=2}^{\infty}$ is a geometric sequence, so we can get
\begin{equation*}
	\mathcal{G}^1_{hist,j}(t_n)=\frac{s^\alpha_j+1}{2}\mathcal{G}^1_{hist,j}(t_{n-1})+w^\alpha_jG(t_{n-2}).
\end{equation*}
Thus we can get $\mathcal{G}^1_{hist,j}(t_n)$ from $\mathcal{G}^1_{hist,j}(t_{n-1})$ and $G(t_{n-2})$ instead of calculating the sum of $\bar{d}_{1,n-i}(s^{\alpha}_j)G(t_i)$. 
So the computation time is reduced from $\mathcal{O}(N^2)$ to $\mathcal{O}(NN_p)$ and the total memory requirement is cut down from  $\mathcal{O}(N)$ to $\mathcal{O}(N_p)$, where $N$ stands for the total number of time steps and $N_p$ is the number of integration points.

\subsection{Fast SBD discretization}
In this subsection, we take $z=\delta(\zeta)=((1-\zeta)+(1-\zeta)^2/2)/\tau$ in Eq. \eqref{zzzz}, which leads to
\begin{equation}\label{equO2weightint}
	\left (\frac{(1-\zeta)+(1-\zeta)^2/2}{\tau}\right )^{\alpha}=\frac{1}{\Gamma(1-\alpha)\Gamma(\alpha)}\int_0^{\infty}s^{\alpha-1} \left (1-\frac{2\tau s}{2\tau s+(1-\zeta)(3-\zeta)}\right )ds.
\end{equation}
Set
\begin{equation}\label{equdefofdj2}
	\left (\frac{(1-\zeta)+(1-\zeta)^2/2}{\tau}\right )^{\alpha}=\sum_{i=0}^{\infty}d^{\alpha}_{2,i}\zeta^i.
\end{equation}
By simple calculation, we have
\begin{equation*}
	\begin{aligned}
		&-\frac{2\tau s}{2\tau s+(1-\zeta)(3-\zeta)}
		=\frac{2\tau s}{2+2\sigma}\left (\frac{1}{\sigma-1+\zeta}+\frac{1}{3+\sigma-\zeta}\right )\\
		&\qquad=\frac{\tau s}{1+\sigma}\left (\frac{1}{3+\sigma}\sum_{i=0}^{\infty}\left(\frac{\zeta}{3+\sigma}\right)^{i}-\frac{1}{1-\sigma}\sum_{i=0}^{\infty}\left(\frac{\zeta}{1-\sigma}\right)^{i}\right ),
	\end{aligned}
\end{equation*}
where $\sigma$ is the solution of  $\sigma^2+2\sigma+2\tau s=0$. Then from \eqref{equO2weightint} we obtain
\begin{equation*}
		d^{\alpha}_{2,0}=\frac{1}{\Gamma(1-\alpha)\Gamma(\alpha)}\int_0^{\infty}s^{\alpha-1} \left (1+\frac{\tau s}{1+\sigma}\left(\frac{1}{3+\sigma}-\frac{1}{1-\sigma}\right)\right )ds,
\end{equation*}
\begin{equation}\label{equO2intrep}
		d^{\alpha}_{2,i}=\frac{1}{\Gamma(1-\alpha)\Gamma(\alpha)}\int_0^{\infty}s^{\alpha-1} \frac{\tau s}{1+\sigma}\left(\left(\frac{1}{3+\sigma}\right)^{i+1}-\left(\frac{1}{1-\sigma}\right)^{i+1}\right)ds,\qquad i\geq 1.
\end{equation}
For $i\geq 3$, by simple calculations and Jordan's Lemma (see Appendix \ref{AP2}), we have
\begin{equation}\label{equO2toderinAp}
	\begin{aligned}
		d^{\alpha}_{2,i}=&-\frac{2^{2+2\alpha}(-1)^{-\alpha}}{\tau^{\alpha}\Gamma(1-\alpha)\Gamma(\alpha)}\int_{-1}^{1}(1-s)^{\alpha}(1+s)^{2-2\alpha}\bar{d}^1_{2,i}(s)ds\\
		&-\frac{2^{-\alpha-3}}{\tau^{\alpha}\Gamma(1-\alpha)\Gamma(\alpha)}\int_{-1}^{1}(1-s)^{\alpha}(1+s)^{2-2\alpha}\bar{d}^2_{2,i}(s)ds,
	\end{aligned}
\end{equation}
where
\begin{equation}\label{equdefofbardj2}
	\begin{aligned}
		&\bar{d}^1_{2,i}(s)=\left(\frac{1}{s+5}\right)^{4}\left(\frac{s+1}{s+5}\right)^{i-3},\\
		&\bar{d}^2_{2,i}(s)=\left(1+3s\right)^{\alpha}\left(\frac{s+1}{2}\right)^{i-3},\qquad i=3,4,\cdots.
	\end{aligned}
\end{equation}
Thus $d^\alpha_{2,i}$ can be approximated as
\begin{equation}\label{eqintegraltoapproO2}
	d^{\alpha}_{2,i}=\sum_{j=1}^{N_{p,1}}w^{\alpha}_{1,j}\bar{d}^{1}_{2,i}(s^{\alpha}_{1,j})+\sum_{j=1}^{N_{p,2}}w^{\alpha}_{2,j}\bar{d}^2_{2,i}(s^{\alpha}_{2,j})+\epsilon^{\alpha}_{2,i}, \qquad i=3,4,\cdots,
\end{equation}
where $\{w^{\alpha}_{1,j}\}_{j=1}^{N_{p,1}}$, $\{w^{\alpha}_{2,j}\}_{j=1}^{N_{p,2}}$ denote the integration weights, $\{s^{\alpha}_{1,j}\}_{j=1}^{N_{p,1}}$, $\{s^{\alpha}_{2,j}\}_{j=1}^{N_{p,2}}$ are the integral points, $N_{p,1}$, $N_{p,2}$ signify the number of integration points, and $\{\epsilon^{\alpha}_{2,i}\}_{i=3}^{\infty}$ indicate the errors caused by the integration approximation. For convenience, we set $\epsilon^{\alpha}_{2,0}=\epsilon^{\alpha}_{2,1}=\epsilon^{\alpha}_{2,2}=0$. To make $|\epsilon^\alpha_{2,i}|$ small enough, the Gauss-Jacobi rule can be used to obtain $w^{\alpha}_{1,j}$, $w^{\alpha}_{2,j}$ and $s^{\alpha}_{1,j}$, $s^{\alpha}_{2,j}$.

\begin{remark}
	According to \eqref{equdefofdj2}, the weights $d^\alpha_{2,i}~ (i=0,1,\cdots)$ are real numbers, but we find that $\bar{d}^\alpha_{2,i}(s)$ for $s<-\frac{1}{3}$ and $w^\alpha_{1,j}$ are complex numbers,  which are caused by two terms $(1+3s)^{\alpha}$ and $(-1)^{-\alpha}$,  respectively. So, to reduce the computation time but without losing precision, we use the real parts of  $\bar{d}^\alpha_{2,i}(s)$ and $w^\alpha_{1,j}$ to accomplish the simulation.
\end{remark}

Then, the Riemann-Liouville fractional derivative can be discretized as
\begin{equation}\label{equrelationO2}
	\begin{aligned}
		\!_0D^{\alpha}_t G(t_n)\approx&\sum_{i=0}^{n}d^{\alpha}_{2,i}G(t_{n-i})\quad {\rm (classical~SBD~discretization)}\\
		=&\sum_{i=0}^{\min(N_s-1,n)}d^{\alpha}_{2,i} G(t_{n-1})+\sum_{i=N_s}^{n}\epsilon^{\alpha}_{2,i}G(t_{n-i})\\
		&+\sum_{i=N_s}^{n}\sum_{j=1}^{N_{p,1}}w^{\alpha}_{1,j}\bar{d}^1_{2,i}(s^{\alpha}_{1,j})G(t_{n-i})+\sum_{i=N_s}^{n}\sum_{j=1}^{N_{p,2}}w^{\alpha}_{2,j}\bar{d}^2_{2,i}(s^{\alpha}_{2,j})G(t_{n-i}),
	\end{aligned}
\end{equation}
where $N_s$ is a parameter that ensures the accuracy of the discretization.
\begin{remark}
Here the reason that we introduce the parameter $N_s$ is that $d_{2,i}^{\alpha}$ can't be approximated effectively by the Gauss-Jacobi rule when $i$ is small, so we start from the $N_s$-th term to approximate $d_{2,i}^{\alpha}$, i.e., we still use the first $N_s$ weights $d_{2,i}^{\alpha}$ generated by convolution quadrature in the fast SBD discretization and $\epsilon^\alpha_{2,i}=0$ for $i<N_s$. The detailed discussions on the value of $N_s$ will be presented in the numerical experiments.
\end{remark}
To get the fast evaluation of \eqref{equrelationO2}, we rewrite it as

\begin{equation*}
	\begin{aligned}
		\!_0D^{\alpha}_t G(t_n)\approx&\sum_{i=0}^{\min(N_s-1,n)}d^{\alpha}_{2,i} G(t_{n-i})\\
		&+\sum_{i=N_s}^{n}\sum_{j=1}^{N_{p,1}}w^{\alpha}_{1,j}\bar{d}^1_{2,i}(s^{\alpha}_{1,j})G(t_{n-i})+\sum_{i=N_s}^{n}\sum_{j=1}^{N_{p,2}}w^{\alpha}_{2,j}\bar{d}^2_{2,i}(s^{\alpha}_{2,j})G(t_{n-i})\\
		=& \sum_{i=0}^{\min(N_s-1,n)}d^{\alpha}_{2,i} G(t_{n-i})\\
		&+\sum_{j=1}^{N_{p,1}}\mathcal{G}^{2,1}_{hist,j}(t_n)+\sum_{j=1}^{N_{p,2}}\mathcal{G}^{2,2}_{hist,j}(t_n), \qquad n=2,3,\ldots,
	\end{aligned}
\end{equation*}
where
\begin{equation}\label{equdefO2historyterm}
	\mathcal{G}^{2,1}_{hist,j}(t_n)=w^{\alpha}_{1,j}\sum_{i=N_s}^{n}\bar{d}^{1}_{2,i}(s^{\alpha}_{1,j})G(t_{n-i}),\qquad \mathcal{G}^{2,2}_{hist,j}(t_n)=w^{\alpha}_{2,j}\sum_{i=N_s}^{n}\bar{d}^{2}_{2,i}(s^{\alpha}_{2,j})G(t_{n-i}),
\end{equation}
and we also call the two terms history parts. Using the property of geometrical sequences $\{\bar{d}^{1}_{2,i}(s)\}_{i=2}^{\infty}$ and $\{\bar{d}^{2}_{2,i}(s)\}_{i=2}^{\infty}$ defined in \eqref{equdefofbardj2}, we obtain
\begin{equation*}
	\begin{aligned}
		&\mathcal{G}^{2,1}_{hist,j}(t_n)=\frac{s^{\alpha}_{1,j}+1}{s^{\alpha}_{1,j}+5}\mathcal{G}^{2,1}_{hist,j}(t_{n-1})+w^{\alpha}_{1,j}\left (\frac{1}{s^{\alpha}_{1,j}+5}\right )^4\left(\frac{s^{\alpha}_{1,j}+1}{s^{\alpha}_{1,j}+5}\right)^{N_s-3}G(t_{n-N_s}),\\
		&\mathcal{G}^{2,2}_{hist,j}(t_n)=\frac{1+s^{\alpha}_{2,j}}{2}\mathcal{G}^{2,2}_{hist,j}(t_{n-1})+w^{\alpha}_{2,j}\left (1+3s^{\alpha}_{2,j}\right )^{\alpha}\left(\frac{1+s^{\alpha}_{2,j}}{2}\right)^{N_s-3} G(t_{n-N_s}).
	\end{aligned}
\end{equation*}
Thus we can get $\mathcal{G}^{2,1}_{hist,j}(t_n)$ and $\mathcal{G}^{2,2}_{hist,j}(t_n)$ from $\mathcal{G}^{2,1}_{hist,j}(t_{n-1})$, $\mathcal{G}^{2,2}_{hist,j}(t_{n-1})$ and $G(t_{n-N_s})$ instead of calculating the sum of $\bar{d}^{1}_{2,n-i}(s^{\alpha}_{1,j})G(t_i)$ and the sum of $\bar{d}^{2}_{2,n-i}(s^{\alpha}_{2,j})G(t_i)$. So the computation time is reduced from $\mathcal{O}(N^2)$ to $\mathcal{O}(N(N_{p,1}+N_{p,2}+N_s))$ and the total memory requirement is reduced from  $\mathcal{O}(N)$ to $\mathcal{O}(N_{p,1}+N_{p,2}+N_s)$, where $N$ is the total number of the time steps, $N_{p,1}$ and $N_{p,2}$ are the number of the integral points, and $N_s$ is a parameter that ensures the accuracy of the approximation.

\section{Error analysis} \label{sec:4}
Now, we apply the fast BE and SBD algorithms developed in Section \ref{sec:3} to solve the system of fractional partial differential equations \eqref{equrqtosol} and give error analyses of the fast BE and SBD schemes, respectively.
\subsection{Error estimates for the fast BE scheme}
According to \cite{Nie2018}, we have the following BE scheme
\begin{equation}\label{equO1tradis}
	\left \{\begin{aligned}
		&\frac{\bar{G}^n_{1}-\bar{G}^{n-1}_{1}}{\tau}+a\sum_{i=0}^{n-1}d^{1-\alpha_1}_{1,i}\bar{G}^{n-i}_{1}+
		\sum_{i=0}^{n-1}d^{1-\alpha_1}_{1,i}A \bar{G}^{n-i}_{1}=a\sum_{i=0}^{n-1}d^{1-\alpha_2}_{1,i}\bar{G}^{n-i}_{2}\qquad {\rm in}~\Omega,~n\geq 1,\\
		&\frac{\bar{G}^n_{2}-\bar{G}^{n-1}_{2}}{\tau}+a\sum_{i=0}^{n-1}d^{1-\alpha_2}_{1,i}\bar{G}^{n-i}_{2}+
		\sum_{i=0}^{n-1}d^{1-\alpha_2}_{1,i}A \bar{G}^{n-i}_{2}=a\sum_{i=0}^{n-1}d^{1-\alpha_1}_{1,i}\bar{G}^{n-i}_{1}\qquad {\rm in}~\Omega,~n\geq 1,\\
		&\bar{G}^0_{1}=G_{1}(0),\quad \bar{G}^0_{2}=G_{2}(0)\qquad\qquad\qquad\qquad\qquad\qquad\qquad\qquad\qquad\qquad\ \ {\rm in}~\Omega,\\ 		&\bar{G}^n_{1}=\bar{G}^n_{2}=0\qquad\qquad\qquad\qquad\qquad\qquad\qquad\qquad\qquad\qquad\qquad\qquad\qquad\: \: {\rm on}~\partial\Omega,~n\geq 0,
	\end{aligned}\right .
\end{equation}
where $ \bar{G}^n_{1} $, $\bar{G}^n_{2}$ are the numerical solutions of $G_1$, $G_2$ at time $t_n$. According to \eqref{equO1fastRL}, we modify the system \eqref{equO1tradis} as the fast BE scheme, i.e.,
\begin{equation}\label{equO1fastdispre}
	\left \{\begin{aligned}
		&\frac{G^n_{1}-G^{n-1}_{1}}{\tau}+d^{1-\alpha_1}_{1,0}(aG^n_1+AG^n_1)+d^{1-\alpha_1}_{1,1}(aG^{n-1}_1+AG^{n-1}_1)\\
		&+a\sum_{i=2}^{n-1}\sum_{j=0}^{N_p}w^{1-\alpha_1}_{j}\bar{d}_{1,i}(s^{1-\alpha_1}_{j})G^{n-i}_{1}+
		\sum_{i=2}^{n-1}\sum_{j=0}^{N_p}w^{1-\alpha_1}_{j}\bar{d}_{1,i}(s^{1-\alpha_1}_{j})A G^{n-i}_{1}=\\
		&\qquad d^{1-\alpha_2}_{1,0}aG^n_2+d^{1-\alpha_2}_{1,1}aG^{n-1}_2+ a\sum_{i=2}^{n-1}\sum_{j=0}^{N_p}w^{1-\alpha_2}_{j}\bar{d}_{1,i}(s^{1-\alpha_2}_{j})G^{n-i}_{2}\qquad {\rm in}~\Omega,~n\geq 1,\\
		&\frac{G^n_{2}-G^{n-1}_{2}}{\tau}+d^{1-\alpha_2}_{1,0}(aG^n_2+AG^n_2)+d^{1-\alpha_2}_{1,1}(aG^{n-1}_2+AG^{n-1}_2)\\
		&+a\sum_{i=2}^{n-1}\sum_{j=0}^{N_p}w^{1-\alpha_2}_{j}\bar{d}_{1,i}(s^{1-\alpha_2}_{j})G^{n-i}_{2}+
		\sum_{i=2}^{n-1}\sum_{j=0}^{N_p}w^{1-\alpha_2}_{j}\bar{d}_{1,i}(s^{1-\alpha_2}_{j})A G^{n-i}_{2}=\\
		&\qquad d^{1-\alpha_1}_{1,0}aG^n_1+d^{1-\alpha_1}_{1,1}aG^{n-1}_1+a\sum_{i=2}^{n-1}\sum_{j=0}^{N_p}w^{1-\alpha_1}_{j}\bar{d}_{1,i}(s^{1-\alpha_1}_{j})G^{n-i}_{1}\qquad {\rm in}~\Omega,~n\geq 1,\\
		&G^0_{1}=G_{1}(0),\quad G^0_{2}=G_{2}(0)\qquad\qquad\qquad\qquad\qquad\qquad\qquad\qquad\qquad\:  {\rm in}~\Omega,\\
		&G^n_{1}=G^n_{2}=0\qquad\qquad\qquad\qquad\qquad\qquad\qquad\qquad\qquad\qquad\qquad\qquad  {\rm on}~\partial\Omega,~n\geq 0.
	\end{aligned}\right .
\end{equation}
where $ G^n_{1} $, $G^n_{2}$ are the numerical solutions of $G_1$, $G_2$ at time $t_n$. From Eq. \eqref{equrelationO1}, the system  \eqref{equO1fastdispre} can be rewritten as
\begin{equation}\label{equO1fastdis}
	\left \{\begin{aligned}
		&\frac{G^n_{1}-G^{n-1}_{1}}{\tau}+a\sum_{i=0}^{n-1}d^{1-\alpha_1}_{1,i}G^{n-i}_{1}+
		\sum_{i=0}^{n-1}d^{1-\alpha_1}_{1,i}A G^{n-i}_{1}=a\sum_{i=0}^{n-1}d^{1-\alpha_2}_{1,i}G^{n-i}_{2}\\
		&\qquad\qquad+\sum_{i=0}^{n-1}\epsilon^{1-\alpha_1}_{1,i}(aG^{n-i}_1+AG^{n-i}_1)-\sum_{i=0}^{n-1}\epsilon^{1-\alpha_2}_{1,i}aG^{n-i}_2\qquad {\rm in}~\Omega,~n\geq 1,\\
		&\frac{G^n_{2}-G^{n-1}_{2}}{\tau}+a\sum_{i=0}^{n-1}d^{1-\alpha_2}_{1,i}G^{n-i}_{2}+
		\sum_{i=0}^{n-1}d^{1-\alpha_2}_{1,i}A G^{n-i}_{2}=a\sum_{i=0}^{n-1}d^{1-\alpha_1}_{1,i}G^{n-i}_{1}\\
		&\qquad\qquad+\sum_{i=0}^{n-1}\epsilon^{1-\alpha_2}_{1,i}(aG^{n-i}_2+AG^{n-i}_2)-\sum_{i=0}^{n-1}\epsilon^{1-\alpha_1}_{1,i}aG^{n-i}_1\qquad {\rm in}~\Omega,~n\geq 1,\\
		&G^0_{1}=G_{1}(0),\quad G^0_{2}=G_{2}(0)\qquad\qquad\qquad\qquad\qquad\qquad\qquad\quad\:\:  {\rm in}~\Omega,\\
		&G^n_{1}=G^n_{2}=0\qquad\qquad\qquad\qquad\qquad\qquad\qquad\qquad\qquad\quad\qquad\:  {\rm on}~\partial\Omega,~n\geq 0.
	\end{aligned}\right .
\end{equation}
\begin{remark}
	Following \cite{Lubich1996}, we omit $G(t_0)$ in \eqref{equrelationO1} when we construct the BE scheme of system \eqref{equrqtosol}, which is helpful in the approximation of the Riemann-Liouville fractional derivative in the system \eqref{equO1fastdis}. Accordingly, the history part defined in \eqref{equdefO1historyterm} is modified as
	\begin{equation*}
	\mathcal{G}^1_{hist,j}(t_n)=w^{\alpha}_j\sum_{i=1}^{n-2}\bar{d}_{1,n-i}(s^{\alpha}_j)G(t_i).
	\end{equation*}
%
\end{remark}
\begin{lemma}[\cite{Jin2017}]\label{lemdelta}
	For any $\theta\in (\pi/2,\pi)$, where $\theta=\arg(z)$, there exist positive constants $c_1$, $c_2$ such that
	\begin{equation*}
		c_1|z|\leq|\delta(e^{-z\tau})|\leq c_2|z|,
	\end{equation*}
	where $\delta(\zeta)=(1-\zeta)/\tau$ or $\delta(\zeta)=((1-\zeta)+(1-\zeta)^2/2)/\tau$.
\end{lemma}

According to \cite{Nie2018}, we have the following estimates between \eqref{equrqtosol} and \eqref{equO1tradis}.
\begin{theorem}\label{thmO1traest}
	Let $G_{1}$, $G_{2}$ and $\bar{G}^n_{1}$, $\bar{G}^n_{2}$ be, respectively, the solutions of the systems \eqref{equrqtosol} and \eqref{equO1tradis}. Then
	\begin{equation*}
		\begin{aligned}
			\|G_{1}(t_n)-\bar{G}^n_{1}\|_{L^2(\Omega)}\leq C\tau\left(t_n^{-1}\|G_{1}(0)\|_{L^2(\Omega)}+t_n^{\alpha_2-1}\|G_{2}(0)\|_{L^2(\Omega)}\right),\\
			\|G_{2}(t_n)-\bar{G}^n_{2}\|_{L^2(\Omega)}\leq C\tau\left(t_n^{\alpha_1-1}\|G_{1}(0)\|_{L^2(\Omega)}+t_n^{-1}\|G_{2}(0)\|_{L^2(\Omega)}\right).
		\end{aligned}
	\end{equation*}
\end{theorem}

Now we provide the regularity of solutions for the system \eqref{equO1tradis}.
\begin{theorem}\label{thmO1trareg}
	Let $\bar{G}^n_{1}$, $\bar{G}^n_{2}$ be the solutions of the system \eqref{equO1tradis}. Then we have
	\begin{equation*}
		\begin{aligned} 
			&\|A^\nu \bar{G}^n_1\|_{L^2(\Omega)}\leq C\left (t^{-\nu\alpha_1}\|G_{1}(0)\|_{L^2(\Omega)}+\|G_{2}(0)\|_{L^2(\Omega)}\right ),\\
			&\|A^\nu \bar{G}^n_2\|_{L^2(\Omega)}\leq C\left (\|G_{1}(0)\|_{L^2(\Omega)}+t^{-\nu\alpha_2}\|G_{2}(0)\|_{L^2(\Omega)}\right )
		\end{aligned}
	\end{equation*}
	for $ \nu=0,1 $; furthermore, if $G_{1}(0),G_{2}(0)\in H^1_0(\Omega)\bigcap H^2(\Omega)$, there exist
	\begin{equation*}
		\begin{aligned}
			&\|A \bar{G}^n_1\|_{L^2(\Omega)}\leq C\left (\|AG_{1}(0)\|_{L^2(\Omega)}+\|AG_{2}(0)\|_{L^2(\Omega)}\right ),\\
			&\|A \bar{G}^n_2\|_{L^2(\Omega)}\leq C\left (\|AG_{1}(0)\|_{L^2(\Omega)}+\|AG_{2}(0)\|_{L^2(\Omega)}\right ).
		\end{aligned}
	\end{equation*}
\end{theorem}
\begin{proof}
	Here we take $\kappa\geq 1/t$ for given $t$ and ensure that $\kappa$ is large enough to satisfy the conditions in Lemmas \ref{lemtheestimateofHblabla} and  \ref{lemaAH}. To get the solutions of the system \eqref{equO1tradis}, we multiply by $\zeta^n$ and sum from $1$ to $\infty$ for both sides of the first two equations in \eqref{equO1tradis} and get
	\begin{equation*}
		\begin{aligned}
			&\sum_{n=1}^{\infty}\frac{\zeta^n \bar{G}^n_{1}-\zeta^n \bar{G}^{n-1}_{1}}{\tau}+a\sum_{n=1}^{\infty}\sum_{i=0}^{n-1}d^{1-\alpha_1}_{1,i}\zeta^n\bar{G}^{n-i}_{1}
			+\sum_{n=1}^{\infty}\sum_{i=0}^{n-1}d^{1-\alpha_1}_{1,i}\zeta^nA \bar{G}^{n-i}_{1}\\
			&\quad\quad\quad\quad\quad\quad\quad\quad\quad\quad\quad\quad\quad\quad\quad=a\sum_{n=1}^{\infty}\sum_{i=0}^{n-1}d^{1-\alpha_2}_{1,i}\zeta^n\bar{G}^{n-i}_{2},\\
			&\sum_{n=1}^{\infty}\frac{\zeta^n \bar{G}^n_{2}-\zeta^n \bar{G}^{n-1}_{2}}{\tau}+a\sum_{n=1}^{\infty}\sum_{i=0}^{n-1}d^{1-\alpha_2}_{1,i}\zeta^n\bar{G}^{n-i}_{2}+
			\sum_{n=1}^{\infty}\sum_{i=0}^{n-1}d^{1-\alpha_2}_{1,i}\zeta^nA \bar{G}^{n-i}_{2}\\
			&\quad\quad\quad\quad\quad\quad\quad\quad\quad\quad\quad\quad\quad\quad\quad=a\sum_{n=1}^{\infty}\sum_{i=0}^{n-1}d^{1-\alpha_1}_{1,i}\zeta^n\bar{G}^{n-i}_{1}.\\
		\end{aligned}
	\end{equation*}
	According to \eqref{equdefofdj1}, we have
	\begin{equation*}
		\begin{aligned}
			&\left (\frac{1-\zeta}{\tau}\right )\sum_{i=1}^{\infty}\bar{G}^i_{1}\zeta^i+a\left (\frac{1-\zeta}{\tau}\right )^{1-\alpha_1}\sum_{i=1}^{\infty}\bar{G}^i_{1}\zeta^i+\left (\frac{1-\zeta}{\tau}\right )^{1-\alpha_1}A\sum_{i=1}^{\infty}\bar{G}^i_{1}\zeta^i\\=&a\left (\frac{1-\zeta}{\tau}\right )^{1-\alpha_2}\sum_{i=1}^{\infty}\bar{G}^i_{2}\zeta^i+\frac{\zeta G_{1}(0)}{\tau},\\
			&\left (\frac{1-\zeta}{\tau}\right )\sum_{i=1}^{\infty}\bar{G}^i_{2}\zeta^i+a\left (\frac{1-\zeta}{\tau}\right )^{1-\alpha_2}\sum_{i=1}^{\infty}\bar{G}^i_{2}\zeta^i+\left (\frac{1-\zeta}{\tau}\right )^{1-\alpha_2}A\sum_{i=1}^{\infty}\bar{G}^i_{2}\zeta^i\\=&a\left (\frac{1-\zeta}{\tau}\right )^{1-\alpha_1}\sum_{i=1}^{\infty}\bar{G}^i_{1}\zeta^i+\frac{\zeta G_{2}(0)}{\tau},
		\end{aligned}
	\end{equation*}
	which result in, after simple calculations,
	\begin{equation}\label{equnumsollapformtoest}
		\begin{aligned}
			\sum_{i=1}^{\infty}\bar{G}^i_{1}\zeta^i=&\frac{\zeta}{\tau}\left (H_{\alpha_1}\left(\frac{1-\zeta}{\tau}\right)\left (\frac{1-\zeta}{\tau}\right )^{\alpha_1-1}G_{1}(0)\right.\\
			&\left.\qquad+aH\left(\frac{1-\zeta}{\tau}\right)\left (\frac{1-\zeta}{\tau}\right )^{\alpha_1-1}G_{2}(0)\right ),\\
			\sum_{i=1}^{\infty}\bar{G}^i_{2}\zeta^i=&\frac{\zeta}{\tau}\left (aH\left(\frac{1-\zeta}{\tau}\right)\left (\frac{1-\zeta}{\tau}\right )^{\alpha_2-1}G_{1}(0)\right.\\
			&\left.\quad+H_{\alpha_2}\left(\frac{1-\zeta}{\tau}\right)\left (\frac{1-\zeta}{\tau}\right )^{\alpha_2-1}G_{2}(0)\right ),
		\end{aligned}
	\end{equation}
	where $H$, $H_{\alpha_1}$, and $H_{\alpha_2}$ are defined by \eqref{equdefHz} and \eqref{H12}.
	By \eqref{equnumsollapformtoest}, for $\xi_\tau=e^{-\tau(\kappa+1)}$, there is
	\begin{equation*}
		\begin{aligned}
			\bar{G}^n_{1}=&\frac{1}{2\pi \mathbf{i}\tau}\int_{|\zeta|=\xi_\tau}\zeta^{-n-1}\zeta\times\left (H_{\alpha_1}\left(\frac{1-\zeta}{\tau}\right)\left (\frac{1-\zeta}{\tau}\right )^{\alpha_1-1}G_{1}(0)\right.\\&\qquad\qquad\qquad\qquad\qquad\left.+aH\left(\frac{1-\zeta}{\tau}\right)\left (\frac{1-\zeta}{\tau}\right )^{\alpha_1-1}G_{2}(0)\right )d\zeta.
		\end{aligned}
	\end{equation*}
	Taking $\zeta=e^{-z\tau}$ leads to
	\begin{equation*}
		\begin{aligned}
			\bar{G}^n_{1}=\frac{1}{2\pi \mathbf{i}}&\int_{\Gamma^\tau}e^{zt_n}e^{-z\tau}\times\left (H_{\alpha_1}\left(\frac{1-e^{-z\tau}}{\tau}\right)\left (\frac{1-e^{-z\tau}}{\tau}\right )^{\alpha_1-1}G_{1}(0)\right.\\&\qquad\qquad\qquad\left.+aH\left(\frac{1-e^{-z\tau}}{\tau}\right)\left (\frac{1-e^{-z\tau}}{\tau}\right )^{\alpha_1-1}G_{2}(0)\right )dz,
		\end{aligned}
	\end{equation*}
	where $\Gamma^\tau=\{z=\kappa+1+\mathbf{i}y:y\in\mathbb{R}~{\rm and}~|y|\leq \pi/\tau\}$. Next,  we deform the contour $\Gamma^\tau$ to
	$\Gamma^\tau_{\theta,\kappa}=\{z\in \mathbb{C}:\kappa\leq |z|\leq\frac{\pi}{\tau\sin(\theta)},|\arg z|=\theta\}\bigcup\{z\in \mathbb{C}:|z|=\kappa,|\arg z|\leq\theta\}$. Thus
	\begin{equation}\label{equtrasemidissolG1}
		\begin{aligned}
			\bar{G}^n_{1}=&\frac{1}{2\pi \mathbf{i}}\int_{\Gamma^\tau_{\theta,\kappa}}e^{zt_n}e^{-z\tau}\times\left (H_{\alpha_1}\left(\frac{1-e^{-z\tau}}{\tau}\right)\left (\frac{1-e^{-z\tau}}{\tau}\right )^{\alpha_1-1}G_{1}(0)\right.\\&\left.\qquad\qquad\qquad\qquad+aH\left(\frac{1-e^{-z\tau}}{\tau}\right)\left (\frac{1-e^{-z\tau}}{\tau}\right )^{\alpha_1-1}G_{2}(0)\right )dz.
		\end{aligned}
	\end{equation}
	Combining Lemmas \ref{lemtheestimateofHblabla} and \ref{lemdelta}, we obtain
	\begin{equation*}
		\begin{aligned}
			\|\bar{G}^n_1\|_{L^2(\Omega)}\leq& C\int_{\Gamma^{\tau}_{\theta,\kappa}}e^{\Re(z)t_n}\left (|z|^{-1}\|G_1(0)\|_{L^2(\Omega)}+a|z|^{-\alpha_2-1}\|G_2(0)\|_{L^2(\Omega)}\right )|dz|\\
			\leq&C\|G_1(0)\|_{L^2(\Omega)}+C\|G_2(0)\|_{L^2(\Omega)},
		\end{aligned}
	\end{equation*}
	where $\Re(z)$ denotes the real part of $z$. Similarly, we have
	\begin{equation*}
		\begin{aligned}
			\|\bar{G}^n_2\|_{L^2(\Omega)}
			\leq C\left(\|G_1(0)\|_{L^2(\Omega)}+\|G_2(0)\|_{L^2(\Omega)}\right).
		\end{aligned}
	\end{equation*}
	Multiplying the operator $A$ on both sides of \eqref{equtrasemidissolG1}, we have
	\begin{equation*}
		\begin{aligned}
			A\bar{G}^n_{1}=&\frac{1}{2\pi \mathbf{i}}\int_{\Gamma^\tau_{\theta,\kappa}}e^{zt_n}e^{-z\tau}\times A\left (H_{\alpha_1}\left(\frac{1-e^{-z\tau}}{\tau}\right)\left (\frac{1-e^{-z\tau}}{\tau}\right )^{\alpha_1-1}G_{1}(0)\right.\\&\left.\qquad\qquad\qquad\qquad+aH\left(\frac{1-e^{-z\tau}}{\tau}\right)\left (\frac{1-e^{-z\tau}}{\tau}\right )^{\alpha_1-1}G_{2}(0)\right )dz.
		\end{aligned}
	\end{equation*}
	When $G_1(0),G_2(0)\in L^2(\Omega)$, according to Lemmas \ref{lemtheestimateofHblabla} and \ref{lemdelta}, there is
	\begin{equation*}
		\|A\bar{G}^n_1\|_{L^2(\Omega)}\leq C(t^{-\alpha_1}\|G_1(0)\|_{L^2(\Omega)}+\|G_2(0)\|_{L^2(\Omega)}).
	\end{equation*}
	When $G_1(0),G_2(0)\in \dot{H}^2(\Omega)$, we obtain
	\begin{equation*}
		\|A\bar{G}^n_1\|_{L^2(\Omega)}\leq C\left(\|AG_1(0)\|_{L^2(\Omega)}+\|AG_2(0)\|_{L^2(\Omega)}\right).
	\end{equation*}
	Similarly,
	when $G_1(0),G_2(0)\in L^2(\Omega)$, we have
	\begin{equation*}
		\|A\bar{G}^n_2\|_{L^2(\Omega)}\leq C\left(\|G_1(0)\|_{L^2(\Omega)}+t^{-\alpha_2}\|G_2(0)\|_{L^2(\Omega)}\right).
	\end{equation*}
	When $G_1(0),G_2(0)\in \dot{H}^2(\Omega)$, there exists
	\begin{equation*}
		\|A\bar{G}^n_2\|_{L^2(\Omega)}\leq C\left(\|AG_1(0)\|_{L^2(\Omega)}+\|AG_2(0)\|_{L^2(\Omega)}\right).
	\end{equation*}
\end{proof}

\begin{theorem}\label{thmerrortraadnfast}
	Let $\bar{G}^n_{1}$, $\bar{G}^n_{2}$ and $G^n_{1}$, $G^n_{2}$  be, respectively, the solutions of the systems \eqref{equO1tradis} and \eqref{equO1fastdis}. Then
	\begin{equation*}
		\begin{aligned}
			\|G^n_{1}-\bar{G}^n_{1}\|_{L^2(\Omega)}+
			\|G^n_2-\bar{G}^n_{2}\|_{L^2(\Omega)} \leq Cn\max_{0\leq i\leq n}(\epsilon_{i})\left(\|G_1(0)\|_{L^2(\Omega)}+\|G_2(0)\|_{L^2(\Omega)}\right),
		\end{aligned}
	\end{equation*}
	where $\epsilon_{i}=\max(|\epsilon^{1-\alpha_1}_{1,i}|,|\epsilon^{1-\alpha_2}_{1,i}|)$.
\end{theorem}
\begin{proof}
	In this proof, we take $\kappa\geq 1/t$ for given $t$ and ensure that $\kappa$ is large enough to satisfy the conditions in Lemmas \ref{lemtheestimateofHblabla} and  \ref{lemaAH}. Subtracting \eqref{equO1fastdis} from \eqref{equO1tradis} and denoting $e^n_{1}=\bar{G}^n_1-G^n_1$, $e^n_{2}=\bar{G}^n_2-G^n_2$, we have
	\begin{equation}\label{equerorroffast}
		\begin{aligned}
			&\frac{e^n_{1}-e^{n-1}_{1}}{\tau}+a\sum_{i=0}^{n-1}d^{1-\alpha_1}_{1,i}e^{n-i}_{1}+
			\sum_{i=0}^{n-1}d^{1-\alpha_1}_{1,i}A e^{n-i}_{1}=a\sum_{i=0}^{n-1}d^{1-\alpha_2}_{1,i}e^{n-i}_{2}\\
			&\qquad\qquad+\sum_{i=0}^{n-1}\epsilon^{1-\alpha_1}_{1,i}(ae^{n-i}_1+Ae^{n-i}_1)-\sum_{i=0}^{n-1}\epsilon^{1-\alpha_2}_{1,i}ae^{n-i}_2\\
			&\qquad\qquad-\sum_{i=0}^{n-1}\epsilon^{1-\alpha_1}_{1,i}(a\bar{G}^{n-i}_1+A\bar{G}^{n-i}_1)+\sum_{i=0}^{n-1}\epsilon^{1-\alpha_2}_{1,i}a\bar{G}^{n-i}_2,\\
			&\frac{e^n_{2}-e^{n-1}_{2}}{\tau}+a\sum_{i=0}^{n-1}d^{1-\alpha_2}_{1,i}e^{n-i}_{2}+
			\sum_{i=0}^{n-1}d^{1-\alpha_2}_{1,i}A e^{n-i}_{2}=a\sum_{i=0}^{n-1}d^{1-\alpha_1}_{1,i}e^{n-i}_{1}\\
			&\qquad\qquad+\sum_{i=0}^{n-1}\epsilon^{1-\alpha_2}_{1,i}(ae^{n-i}_2+Ae^{n-i}_2)-\sum_{i=0}^{n-1}\epsilon^{1-\alpha_1}_{1,i}ae^{n-i}_1\\
			&\qquad\qquad-\sum_{i=0}^{n-1}\epsilon^{1-\alpha_2}_{1,i}(a\bar{G}^{n-i}_2+A\bar{G}^{n-i}_2)+\sum_{i=0}^{n-1}\epsilon^{1-\alpha_1}_{1,i}a\bar{G}^{n-i}_1.
		\end{aligned}
	\end{equation}
	Multiplying $\zeta^n$ and summing from $1$ to $\infty$ for the both sides of\ equations in \eqref{equerorroffast} lead to
	\begin{equation*}
		\begin{aligned}
			&\sum_{n=1}^{\infty}\frac{e^n_{1}-e^{n-1}_{1}}{\tau}\zeta^n+a\sum_{n=1}^{\infty}\sum_{i=0}^{n-1}d^{1-\alpha_1}_{1,i}e^{n-i}_{1}\zeta^n+
			\sum_{n=1}^{\infty}\sum_{i=0}^{n-1}d^{1-\alpha_1}_{1,i}A e^{n-i}_{1}\zeta^n=a\sum_{n=1}^{\infty}\sum_{i=0}^{n-1}d^{1-\alpha_2}_{1,i}e^{n-i}_{2}\zeta^n\\
			&\qquad\qquad+\sum_{n=1}^{\infty}\sum_{i=0}^{n-1}\epsilon^{1-\alpha_1}_{1,i}(ae^{n-i}_1+Ae^{n-i}_1)\zeta^n-\sum_{n=1}^{\infty}\sum_{i=0}^{n-1}\epsilon^{1-\alpha_2}_{1,i}ae^{n-i}_2\zeta^n\\
			&\qquad\qquad-\sum_{n=1}^{\infty}\sum_{i=0}^{n-1}\epsilon^{1-\alpha_1}_{1,i}(a\bar{G}^{n-i}_1+A\bar{G}^{n-i}_1)\zeta^n+\sum_{n=1}^{\infty}\sum_{i=0}^{n-1}\epsilon^{1-\alpha_2}_{1,i}a\bar{G}^{n-i}_2\zeta^n,\\
			&\sum_{n=1}^{\infty}\frac{e^n_{2}-e^{n-1}_{2}}{\tau}\zeta^n+a\sum_{n=1}^{\infty}\sum_{i=0}^{n-1}d^{1-\alpha_2}_{1,i}e^{n-i}_{2}\zeta^n+
			\sum_{n=1}^{\infty}\sum_{i=0}^{n-1}d^{1-\alpha_2}_{1,i}A e^{n-i}_{2}\zeta^n=a\sum_{n=1}^{\infty}\sum_{i=0}^{n-1}d^{1-\alpha_1}_{1,i}e^{n-i}_{1}\zeta^n\\
			&\qquad\qquad+\sum_{n=1}^{\infty}\sum_{i=0}^{n-1}\epsilon^{1-\alpha_2}_{1,i}(ae^{n-i}_2+Ae^{n-i}_2)\zeta^n-\sum_{n=1}^{\infty}\sum_{i=0}^{n-1}\epsilon^{1-\alpha_1}_{1,i}ae^{n-i}_1\zeta^n\\
			&\qquad\qquad-\sum_{n=1}^{\infty}\sum_{i=0}^{n-1}\epsilon^{1-\alpha_2}_{1,i}(a\bar{G}^{n-i}_2+A\bar{G}^{n-i}_2)\zeta^n+\sum_{n=1}^{\infty}\sum_{i=0}^{n-1}\epsilon^{1-\alpha_1}_{1,i}a\bar{G}^{n-i}_1\zeta^n.
		\end{aligned}
	\end{equation*}
	Introducing
	\begin{equation*}
		\varepsilon^{\alpha}_1(\zeta)=\sum_{i=0}^{\infty}\epsilon^{\alpha}_{1,i}\zeta^i
	\end{equation*}
	and using \eqref{equdefofdj1}, we have
	\begin{equation*}
		\begin{aligned}
			&\left(\frac{1-\zeta}{\tau}\right)\sum_{n=1}^{\infty}e^n_1\zeta^n+a\left(\frac{1-\zeta}{\tau}\right)^{1-\alpha_1}\sum_{n=1}^{\infty}e^{n}_{1}\zeta^n+
			\left(\frac{1-\zeta}{\tau}\right)^{1-\alpha_1}\sum_{n=1}^{\infty}A e^{n}_{1}\zeta^n=a\left(\frac{1-\zeta}{\tau}\right)^{1-\alpha_2}\sum_{n=1}^{\infty}e^{n}_{2}\zeta^n\\
			&\qquad\qquad+\varepsilon^{1-\alpha_1}_1(\zeta)\sum_{n=1}^{\infty}(ae^{n}_1+Ae^{n}_1)\zeta^n-\varepsilon^{1-\alpha_2}_1(\zeta)\sum_{n=1}^{\infty}ae^{n}_2\zeta^n\\
			&\qquad\qquad-\varepsilon^{1-\alpha_1}_1(\zeta)\sum_{n=1}^{\infty}(a\bar{G}^{n}_1+A\bar{G}^{n}_1)\zeta^n+\varepsilon^{1-\alpha_2}_1(\zeta)\sum_{n=1}^{\infty}a\bar{G}^{n}_2\zeta^n,\\
			&\left(\frac{1-\zeta}{\tau}\right)\sum_{n=1}^{\infty}e^n_2\zeta^n+a\left(\frac{1-\zeta}{\tau}\right)^{1-\alpha_2}\sum_{n=1}^{\infty}e^{n}_{2}\zeta^n+
			\left(\frac{1-\zeta}{\tau}\right)^{1-\alpha_2}\sum_{n=1}^{\infty}A e^{n}_{2}\zeta^n=a\left(\frac{1-\zeta}{\tau}\right)^{1-\alpha_1}\sum_{n=1}^{\infty}e^{n}_{1}\zeta^n\\
			&\qquad\qquad+\varepsilon^{1-\alpha_2}_1(\zeta)\sum_{n=1}^{\infty}(ae^{n}_2+Ae^{n}_2)\zeta^n-\varepsilon^{1-\alpha_1}_1(\zeta)\sum_{n=1}^{\infty}ae^{n}_1\zeta^n\\
			&\qquad\qquad-\varepsilon^{1-\alpha_2}_1(\zeta)\sum_{n=1}^{\infty}(a\bar{G}^{n}_2+A\bar{G}^{n}_2)\zeta^n+\varepsilon^{1-\alpha_1}_1(\zeta)\sum_{n=1}^{\infty}a\bar{G}^{n}_1\zeta^n.
		\end{aligned}
	\end{equation*}
Making simple calculations leads to
	\begin{equation*}
		\begin{aligned}
			\sum_{n=1}^{\infty}e^n_1\zeta^n=&H_{\alpha_1}\left (\frac{1-\zeta}{\tau}\right )\left (\frac{1-\zeta}{\tau}\right )^{\alpha_1-1}\times\left(\varepsilon^{1-\alpha_1}_1(\zeta)(a+A)\sum_{n=1}^{\infty}e^{n}_1\zeta^n-\varepsilon^{1-\alpha_2}_1(\zeta)\sum_{n=1}^{\infty}ae^{n}_2\zeta^n\right.\\
			&\left.-\varepsilon^{1-\alpha_1}_1(\zeta)(a+A)\sum_{n=1}^{\infty}\bar{G}^{n}_1\zeta^n+\varepsilon^{1-\alpha_2}_1(\zeta)\sum_{n=1}^{\infty}a\bar{G}^{n}_2\zeta^n\right)\\
			&+aH\left (\frac{1-\zeta}{\tau}\right )\left (\frac{1-\zeta}{\tau}\right )^{\alpha_1-1}\times\left (\varepsilon^{1-\alpha_2}_1(\zeta)(a+A)\sum_{n=1}^{\infty}e^{n}_2\zeta^n-\varepsilon^{1-\alpha_1}_1(\zeta)\sum_{n=1}^{\infty}ae^{n}_1\zeta^n\right .\\
			&\left .-\varepsilon^{1-\alpha_2}_1(\zeta)(a+A)\sum_{n=1}^{\infty}\bar{G}^{n}_2\zeta^n+\varepsilon^{1-\alpha_1}_1(\zeta)\sum_{n=1}^{\infty}a\bar{G}^{n}_1\zeta^n\right ),\\
			\sum_{n=1}^{\infty}e^n_2\zeta^n=&aH\left (\frac{1-\zeta}{\tau}\right )\left (\frac{1-\zeta}{\tau}\right )^{\alpha_2-1}\times\left(\varepsilon^{1-\alpha_1}_1(\zeta)(a+A)\sum_{n=1}^{\infty}e^{n}_1\zeta^n-\varepsilon^{1-\alpha_2}_1(\zeta)\sum_{n=1}^{\infty}ae^{n}_2\zeta^n\right.\\
			&\left.-\varepsilon^{1-\alpha_1}_1(\zeta)(a+A)\sum_{n=1}^{\infty}\bar{G}^{n}_1\zeta^n+\varepsilon^{1-\alpha_2}_1(\zeta)\sum_{n=1}^{\infty}a\bar{G}^{n}_2\zeta^n\right)\\
			&+H_{\alpha_2}\left (\frac{1-\zeta}{\tau}\right )\left (\frac{1-\zeta}{\tau}\right )^{\alpha_2-1}\times\left (\varepsilon^{1-\alpha_2}_1(\zeta)(a+A)\sum_{n=1}^{\infty}e^{n}_2\zeta^n-\varepsilon^{1-\alpha_1}_1(\zeta)\sum_{n=1}^{\infty}ae^{n}_1\zeta^n\right .\\
			&\left .-\varepsilon^{1-\alpha_2}_1(\zeta)(a+A)\sum_{n=1}^{\infty}\bar{G}^{n}_2\zeta^n+\varepsilon^{1-\alpha_1}_1(\zeta)\sum_{n=1}^{\infty}a\bar{G}^{n}_1\zeta^n\right ).
		\end{aligned}
	\end{equation*}
	Denote
	\begin{equation*}
		\begin{aligned}
			&\sum_{i=0}^{\infty}E^i_{0,1,1}\zeta^i=a^2H\left (\frac{1-\zeta}{\tau}\right )\left (\frac{1-\zeta}{\tau}\right )^{\alpha_1-1}\varepsilon^{1-\alpha_1}_1(\zeta),\\
			&\sum_{i=0}^{\infty}E^i_{0,1,2}\zeta^i=aH\left (\frac{1-\zeta}{\tau}\right )\left (\frac{1-\zeta}{\tau}\right )^{\alpha_1-1}(a+A)\varepsilon^{1-\alpha_2}_1(\zeta),\\
			&\sum_{i=0}^{\infty}E^i_{0,2,1}\zeta^i=aH\left (\frac{1-\zeta}{\tau}\right )\left (\frac{1-\zeta}{\tau}\right )^{\alpha_2-1}(a+A)\varepsilon^{1-\alpha_1}_1(\zeta),\\
			&\sum_{i=0}^{\infty}E^i_{0,2,2}\zeta^i=a^2H\left (\frac{1-\zeta}{\tau}\right )\left (\frac{1-\zeta}{\tau}\right )^{\alpha_2-1}\varepsilon^{1-\alpha_2}_1(\zeta),\\
			&\sum_{i=0}^{\infty}E^i_{1,1,1}\zeta^i=H_{\alpha_1}\left (\frac{1-\zeta}{\tau}\right )\left (\frac{1-\zeta}{\tau}\right )^{\alpha_1-1}(a+A)\varepsilon^{1-\alpha_1}_1(\zeta),\\
			&\sum_{i=0}^{\infty}E^i_{1,1,2}\zeta^i=aH_{\alpha_1}\left (\frac{1-\zeta}{\tau}\right )\left (\frac{1-\zeta}{\tau}\right )^{\alpha_1-1}\varepsilon^{1-\alpha_2}_1(\zeta),\\
			&\sum_{i=0}^{\infty}E^i_{1,2,1}\zeta^i=aH_{\alpha_2}\left (\frac{1-\zeta}{\tau}\right )\left (\frac{1-\zeta}{\tau}\right )^{\alpha_2-1}\varepsilon^{1-\alpha_1}_1(\zeta),\\
			&\sum_{i=0}^{\infty}E^i_{1,2,2}\zeta^i=H_{\alpha_2}\left (\frac{1-\zeta}{\tau}\right )\left (\frac{1-\zeta}{\tau}\right )^{\alpha_2-1}(a+A)\varepsilon^{1-\alpha_2}_1(\zeta).
		\end{aligned}
	\end{equation*}
Consider the estimate of $E^n_{0,1,1}$. Taking $\xi_{\tau}=e^{-\tau(\kappa+1)}$ results in
	\begin{equation*}
		E^n_{0,1,1}=\frac{1}{2\pi \mathbf{i}}\int_{|\zeta|=\zeta_{\tau}}\zeta^{-n-1}a^2H\left (\frac{1-\zeta}{\tau}\right )\left (\frac{1-\zeta}{\tau}\right )^{\alpha_1-1}\varepsilon^{1-\alpha_1}_1(\zeta)d\zeta,
	\end{equation*}
	which leads to
	\begin{equation*}
		\|E^n_{0,1,1}\|\leq C\tau^{-1}\max_{0\leq i\leq n}(|\epsilon^{1-\alpha_1}_{1,i}|)\int_{|\zeta|=\zeta_{\tau}}|\zeta|^{-n-1}\left \|a^2H\left (\frac{1-\zeta}{\tau}\right )\right \| \left |\frac{1-\zeta}{\tau}\right |^{\alpha_1-2}|d\zeta|.
	\end{equation*}
Letting $\zeta=e^{-z\tau}$, we obtain
	\begin{equation*}
		\|E^n_{0,1,1}\|\leq C\max_{0\leq i\leq n}(|\epsilon^{1-\alpha_1}_{1,i}|)\int_{\Gamma^{\tau}}|e^{zt_n}|\left \|a^2H\left (\frac{1-e^{-z\tau}}{\tau}\right )\right \| \left |\frac{1-e^{-z\tau}}{\tau}\right |^{\alpha_1-2}|dz|,
	\end{equation*}
	where $\Gamma^\tau=\{z=\kappa+1+\mathbf{i}y:y\in\mathbb{R}~{\rm and}~|y|\leq \pi/\tau\}$. Next we deform the contour $\Gamma^\tau$ to
	$\Gamma^\tau_{\theta,\kappa}=\{z\in \mathbb{C}:\kappa\leq |z|\leq\frac{\pi}{\tau\sin(\theta)},|\arg z|=\theta\}\bigcup\{z\in \mathbb{C}:|z|=\kappa,|\arg z|\leq\theta\}$. Then
	\begin{equation*}
		\|E^n_{0,1,1}\|\leq C\max_{0\leq i\leq n}(|\epsilon^{1-\alpha_1}_{1,i}|)\int_{\Gamma^{\tau}_{\theta,\kappa}}|e^{zt_n}|\left \|a^2H\left (\frac{1-e^{-z\tau}}{\tau}\right )\right \| \left |\frac{1-e^{-z\tau}}{\tau}\right |^{\alpha_1-2}|dz|.
	\end{equation*}
	Using Lemmas \ref{lemtheestimateofHblabla} and \ref{lemdelta}, we have
	\begin{equation*}
		\|E^n_{0,1,1}\|\leq C\max_{0\leq i\leq n}(|\epsilon^{1-\alpha_1}_{1,i}|)\int_{\Gamma^{\tau}_{\theta,\kappa}}e^{\Re(z)t_n}|z|^{-\alpha_2-2}|dz|\leq C\max_{0\leq i\leq n}(\epsilon_{i}).
	\end{equation*}
	Similarly, we have
	\begin{equation*}
		\|E^n_{0,1,2}\|,\|E^n_{0,2,1}\|,\|E^n_{0,2,2}\|,\|E^n_{1,1,1}\|,\|E^n_{1,1,2}\|,\|E^n_{1,2,1}\|,\|E^n_{1,2,2}\|\leq C\max_{0\leq i\leq n}(\epsilon_{i}).
	\end{equation*}
	Thus
	\begin{equation*}
		\begin{aligned}
			\|e^n_1\|_{L^2(\Omega)}+\|e^n_2\|_{L^2(\Omega)}\leq& C\max_{0\leq i\leq n}(\epsilon_{i})\sum_{i=1}^{n-1}(\|e^i_1\|_{L^2(\Omega)}		+\|e^i_2\|_{L^2(\Omega)})\\
			&+C\max_{0\leq i\leq n}(\epsilon_{i})\sum_{i=1}^{n}(\|\bar{G}^i_1\|_{L^2(\Omega)}+\|\bar{G}^i_2\|_{L^2(\Omega)}).
		\end{aligned}
	\end{equation*}
Combining Gr\"{o}nwall's inequality and Theorem \ref{thmO1trareg} leads to
	\begin{equation*}
		\|e^n_1\|_{L^2(\Omega)}+\|e^n_2\|_{L^2(\Omega)}\leq Cn\max_{0\leq i\leq n}(\epsilon_{i})(\|G_1(0)\|_{L^2(\Omega)}+\|G_2(0)\|_{L^2(\Omega)}).
	\end{equation*}
\end{proof}

Combining Theorem  \ref{thmO1traest} and Theorem \ref{thmerrortraadnfast}, we get the following error estimates for the fast BE scheme.
\begin{theorem}
	Let $G_{1}$, $G_{2}$ and $G^n_{1}$, $G^n_{2}$ be the solutions of the systems \eqref{equrqtosol} and \eqref{equO1fastdis}, respectively. Then we have the estimates, if $G_{1}(0)$, $G_{2}(0)\in L^2(\Omega)$,
	\begin{equation*}
		\begin{aligned}
			\|G_{1}(t_n)-G^n_{1}\|_{L^2(\Omega)}\leq&C\tau\left(t_n^{-1}\|G_{1}(0)\|_{L^2(\Omega)}+t_n^{\alpha_2-1}\|G_{2}(0)\|_{L^2(\Omega)}\right)\\
			&+Cn\max_{0\leq i\leq n}(\epsilon_{i})\left(\|G_{1}(0)\|_{L^2(\Omega)}+\|G_{2}(0)\|_{L^2(\Omega)}\right),\\
			\|G_{2}(t_n)-G^n_{2}\|_{L^2(\Omega)}\leq& C\tau\left(t_n^{\alpha_1-1}\|G_{1}(0)\|_{L^2(\Omega)}+t_n^{-1}\|G_{2}(0)\|_{L^2(\Omega)}\right)\\
			&+Cn\max_{0\leq i\leq n}(\epsilon_{i})\left(\|G_{1}(0)\|_{L^2(\Omega)}+\|G_{2}(0)\|_{L^2(\Omega)}\right),
		\end{aligned}
	\end{equation*}
	where $\epsilon_{i}=\max(|\epsilon^{1-\alpha_1}_{1,i}|,|\epsilon^{1-\alpha_2}_{1,i}|)$.	
\end{theorem}

\subsection{Error estimates for the fast SBD scheme}
Here, we first provide the SBD scheme of \eqref{equrqtosol} by SBD convolution quadrature and give the error estimate of the SBD scheme. Then we present the error estimate of the fast SBD scheme. According to \cite{Lubich1996,Jin2016}, to keep the accuracy of the scheme, one needs to modify the discretization \eqref{equrelationO2}. Namely,  denoting $\bar{\partial}^{\alpha}_{\tau}$ as the discretization of $_{0}D^{\alpha}_t$ and letting $\partial^{-1}_t$ be the integration on time, from \eqref{equrqtosol} we obtain
\begin{equation*}
	\left.
	\begin{aligned}
		&\bar{\partial}_{\tau}  \bar{G}_1^n+a\bar{\partial}^{1-\alpha_1}_{\tau}(\bar{G}_1^n-G_1(0))+\bar{\partial}^{1-\alpha_1}_{\tau}A(\bar{G}_1^n-G_1(0))-a\bar{\partial}^{1-\alpha_2}_{\tau}(\bar{G}_2^n-G_2(0))\\
		&\qquad\qquad\qquad=-a\bar{\partial}^{1-\alpha_1}_{\tau}\bar{\partial}_\tau \partial^{-1}_t G_1(0)-\bar{\partial}^{1-\alpha_1}_{\tau}\bar{\partial}_\tau \partial^{-1}_t A G_1(0)+a\bar{\partial}^{1-\alpha_2}_{\tau}\bar{\partial}_\tau \partial^{-1}_t G_2(0),\\
		&\bar{\partial}_{\tau} \bar{G}_2^n+a\bar{\partial}^{1-\alpha_2}_{\tau}(\bar{G}_2^n-G_2(0))+\bar{\partial}^{1-\alpha_2}_\tau A (\bar{G}_2^n-G_2(0))-a\bar{\partial}^{1-\alpha_1}_{\tau}(\bar{G}_1^n-G_1(0))\\
		&\qquad\qquad\qquad=-a\bar{\partial}^{1-\alpha_2}_{\tau}\bar{\partial}_\tau \partial^{-1}_t G_2(0)-\bar{\partial}^{1-\alpha_2}_{\tau}\bar{\partial}_\tau \partial^{-1}_t A G_2(0)+a\bar{\partial}^{1-\alpha_1}_{\tau}\bar{\partial}_\tau \partial^{-1}_t G_1(0).
	\end{aligned}
	\right .
\end{equation*}
According to the fact $(0,3/2,1,1,\ldots)=\bar{\partial}_\tau\partial^{-1}_t1$ \cite{Jin2016} and setting $\bar{G}^{-1}_1=G_1(0)$, $\bar{G}^{-1}_2=G_2(0)$, the SBD scheme of \eqref{equrqtosol} can be written as
\begin{equation}\label{equO2tradis}
	\left \{\begin{aligned}
		&\frac{\bar{G}^1_{1}-\bar{G}^{0}_{1}}{\tau}+ad^{1-\alpha_1}_{2,0}\left(\frac{2}{3}\bar{G}^{1}_{1}+\frac{1}{3}\bar{G}^{0}_{1}\right)+
		d^{1-\alpha_1}_{2,0}A \left(\frac{2}{3}\bar{G}^{1}_{1}+\frac{1}{3}\bar{G}^{0}_{1}\right)=ad^{1-\alpha_2}_{2,0}\left(\frac{2}{3}\bar{G}^{1}_{2}+\frac{1}{3}\bar{G}^{0}_{2}\right)\quad {\rm in}~\Omega,\\
		&\frac{\bar{G}^1_{2}-\bar{G}^{0}_{2}}{\tau}+ad^{1-\alpha_2}_{2,0}\left(\frac{2}{3}\bar{G}^{1}_{2}+\frac{1}{3}\bar{G}^{0}_{2}\right)+
		d^{1-\alpha_2}_{2,0}A \left(\frac{2}{3}\bar{G}^{1}_{2}+\frac{1}{3}\bar{G}^{0}_{2}\right)=ad^{1-\alpha_1}_{2,0}\left(\frac{2}{3}\bar{G}^{1}_{1}+\frac{1}{3}\bar{G}^{0}_{1}\right)\quad {\rm in}~\Omega,\\
		&\frac{1}{\tau}\left(\frac{3}{2}\bar{G}^n_1-2\bar{G}^{n-1}_1+\frac{1}{2}\bar{G}^{n-2}_1\right)+a\sum_{i=0}^{n-1}d^{1-\alpha_1}_{2,i}\bar{G}^{n-i}_{1}+a\frac{1}{2}d^{1-\alpha_1}_{2,n-1}\bar{G}^{0}_{1}\\
		&\qquad
		+\sum_{i=0}^{n-1}d^{1-\alpha_1}_{2,i}A \bar{G}^{n-i}_{1}+\frac{1}{2}d^{1-\alpha_1}_{2,n-1}A \bar{G}^{0}_{1}=a\sum_{i=0}^{n-1}d^{1-\alpha_2}_{2,i}\bar{G}^{n-i}_{2}+\frac{1}{2}ad^{1-\alpha_2}_{2,n-1}\bar{G}^{0}_{2}\qquad\quad\qquad\quad\ \: {\rm in}~\Omega,~ n\geq 2,\\
		&\frac{1}{\tau}\left(\frac{3}{2}\bar{G}^n_2-2\bar{G}^{n-1}_2+\frac{1}{2}\bar{G}^{n-2}_2\right)+a\sum_{i=0}^{n-1}d^{1-\alpha_2}_{2,i}\bar{G}^{n-i}_{2}+a\frac{1}{2}d^{1-\alpha_2}_{2,n-1}\bar{G}^{0}_{2}\\
		&\qquad+\sum_{i=0}^{n-1}d^{1-\alpha_2}_{2,i}A \bar{G}^{n-i}_{2}+\frac{1}{2}d^{1-\alpha_2}_{2,n-1}A \bar{G}^{0}_{2}=a\sum_{i=0}^{n-1}d^{1-\alpha_1}_{2,i}\bar{G}^{n-i}_{1}+\frac{1}{2}ad^{1-\alpha_1}_{2,n-1}\bar{G}^{0}_{1}\qquad\quad\qquad\quad\: \: {\rm in}~\Omega,~ n\geq 2,\\
		&\bar{G}^0_{1}=G_{1}(0),\quad\bar{G}^0_{2}=G_{2}(0)\qquad\qquad\qquad\qquad\quad\qquad\qquad\qquad\qquad\qquad\qquad\qquad\qquad\quad\ \ {\rm in}~\Omega,\\
		&\bar{G}^n_{1}=\bar{G}^n_{2}=0\qquad\qquad\qquad\qquad\qquad\qquad\qquad\qquad\qquad\qquad\qquad\qquad\qquad\qquad\quad\qquad\quad\   {\rm on}~\partial\Omega,~n\geq 0,
	\end{aligned}\right .
\end{equation}
where $\bar{G}^n_1$, $\bar{G}^n_2$ are the numerical solutions of $G_1$, $G_2$ at $t_n$. Similarly using Eq. \eqref{equrelationO2}, and noting that $\epsilon^{\alpha}_{2,0}=\epsilon^{\alpha}_{2,1}=\epsilon^{\alpha}_{2,2}=0$, we obtain the fast SBD scheme
\begin{equation}\label{equO2fastdis}
	\left \{\begin{aligned}
		&\frac{G^1_{1}-G^{0}_{1}}{\tau}+ad^{1-\alpha_1}_{2,0}\left(\frac{2}{3}G^{1}_{1}+\frac{1}{3}G^{0}_{1}\right)+
		d^{1-\alpha_1}_{2,0}A \left(\frac{2}{3}G^{1}_{1}+\frac{1}{3}G^{0}_{1}\right)=ad^{1-\alpha_2}_{2,0}\left(\frac{2}{3}G^{1}_{2}+\frac{1}{3}G^{0}_{2}\right)\quad {\rm in}~\Omega,\\
		&\frac{G^1_{2}-G^{0}_{2}}{\tau}+ad^{1-\alpha_2}_{2,0}\left(\frac{2}{3}G^{1}_{2}+\frac{1}{3}G^{0}_{2}\right)+
		d^{1-\alpha_2}_{2,0}A \left(\frac{2}{3}G^{1}_{2}+\frac{1}{3}G^{0}_{2}\right)=ad^{1-\alpha_1}_{2,0}\left(\frac{2}{3}G^{1}_{1}+\frac{1}{3}G^{0}_{1}\right)\quad {\rm in}~\Omega,\\
		&\frac{1}{\tau}\left(\frac{3}{2}G^n_1-2G^{n-1}_1+\frac{1}{2}G^{n-2}_1\right)+a\sum_{i=0}^{n-1}d^{1-\alpha_1}_{2,i}G^{n-i}_{1}+\frac{1}{2}ad^{1-\alpha_1}_{2,n-1}G^0_1\\
		&\qquad+
		\sum_{i=0}^{n-1}d^{1-\alpha_1}_{2,i}A G^{n-i}_{1}+\frac{1}{2}d^{1-\alpha_1}_{2,n-1}AG^0_1=a\sum_{i=0}^{n-1}d^{1-\alpha_2}_{2,i}G^{n-i}_{2}+\frac{1}{2}ad^{1-\alpha_2}_{2,n-1}G^0_2\\
		&\qquad+\sum_{i=0}^{n-1}\epsilon^{1-\alpha_1}_{2,i}(aG^{n-i}_1+AG^{n-i}_1)-\sum_{i=0}^{n-1}\epsilon^{1-\alpha_2}_{2,i}aG^{n-i}_2\qquad\qquad\qquad\qquad\qquad\quad\qquad\quad {\rm in}~\Omega,~ n\geq 2,\\
		&\frac{1}{\tau}\left(\frac{3}{2}G^n_2-2G^{n-1}_2+\frac{1}{2}G^{n-2}_2\right)+a\sum_{i=0}^{n-1}d^{1-\alpha_2}_{2,i}G^{n-i}_{2}+\frac{1}{2}ad^{1-\alpha_2}_{2,n-1}G^0_2,\\
		&\qquad+
		\sum_{i=0}^{n-1}d^{1-\alpha_2}_{2,i}A G^{n-i}_{2}+\frac{1}{2}d^{1-\alpha_2}_{2,n-1}AG^0_2=a\sum_{i=0}^{n-1}d^{1-\alpha_1}_{2,i}G^{n-i}_{1}+\frac{1}{2}ad^{1-\alpha_1}_{2,n-1}G^0_1\\
		&\qquad+\sum_{i=0}^{n-1}\epsilon^{1-\alpha_2}_{2,i}(aG^{n-i}_2+AG^{n-i}_2)-\sum_{i=0}^{n-1}\epsilon^{1-\alpha_1}_{2,i}aG^{n-i}_1\qquad\qquad\qquad\qquad\qquad\quad\qquad\quad {\rm in}~\Omega,~ n\geq 2,\\
		&G^0_{1}=G_{1}(0),\quad G^0_{2}=G_{2}(0)\qquad\qquad\qquad\qquad\quad\qquad\qquad\qquad\qquad\qquad\qquad\qquad\qquad\   {\rm in}~\Omega,\\
		&G^n_{1}=G^n_{2}=0\qquad\qquad\qquad\qquad\qquad\qquad\qquad\qquad\qquad\qquad\qquad\qquad\qquad\qquad\qquad\quad  {\rm on}~\partial\Omega,~n\geq 0,
	\end{aligned}\right.
\end{equation}
where $G^n_1$, $G^n_2$ are the numerical solutions of $G_1$, $G_2$ at time $t_n$.
\begin{remark}

To keep the accuracy of SBD scheme \eqref{equO2tradis}, the discretization of $\,_0D^\alpha_tG(t)$ should be modified and correspondingly the definition of history parts in \eqref{equdefO2historyterm} should be changed as
	\begin{equation*}
		\mathcal{G}^{2,1}_{hist,j}(t_n)=w^{\alpha}_{1,j}\sum_{i=N_s}^{n-1}\bar{d}^{1}_{2,i}(s^{\alpha}_{1,j})G(t_{n-i}),\qquad \mathcal{G}^{2,2}_{hist,j}(t_n)=w^{\alpha}_{2,j}\sum_{i=N_s}^{n-1}\bar{d}^{2}_{2,i}(s^{\alpha}_{2,j})G(t_{n-i}).
	\end{equation*}
\end{remark}

Next, we provide the error estimates between \eqref{equrqtosol} and \eqref{equO2tradis}.
\begin{theorem}\label{thmO2CQerr}
	Let $G_{1}$, $G_{2}$ and $\bar{G}^n_{1}$, $\bar{G}^n_{2}$ be, respectively, the solutions of the systems \eqref{equrqtosol} and \eqref{equO2tradis}. Then
	\begin{equation*}
		\begin{aligned}
			\|G_{1}(t_n)-\bar{G}^n_{1}\|_{L^2(\Omega)}\leq C\tau^2\left(t_n^{-2}\|G_{1}(0)\|_{L^2(\Omega)}+t_n^{\alpha_2-2}\|G_{2}(0)\|_{L^2(\Omega)}\right),\\
			\|G_{2}(t_n)-\bar{G}^n_{2}\|_{L^2(\Omega)}\leq C\tau^2\left(t_n^{\alpha_1-2}\|G_{1}(0)\|_{L^2(\Omega)}+t_n^{-2}\|G_{2}(0)\|_{L^2(\Omega)}\right).
		\end{aligned}
	\end{equation*}
\end{theorem}
\begin{proof}
	Here we take $\kappa\geq 1/t$ for given $t$ and ensure that $\kappa$ is large enough to satisfy the conditions in Lemmas \ref{lemtheestimateofHblabla} and  \ref{lemaAH}. Introduce $V_1(t)=G_1(t)-G_1(0)$, $V_2(t)=G_2(t)-G_2(0)$, $\bar{V}^n_1=\bar{G}^n_1-G_1(0)$, $\bar{V}^n_2=\bar{G}^n_2-G_2(0)$; and denote $\mathbf{V}=[V_1, V_2]^T$. Thus \eqref{equrqtosol} can be written as
	\begin{equation}\label{equO2retosol}
		\left \{
		\begin{aligned}
			&\frac{\partial V_1}{\partial t}+a~_0D^{1-\alpha_1}_tV_1+~_0D^{1-\alpha_1}_tA V_1-a~_0D^{1-\alpha_2}_tV_2=\\
			&\qquad\qquad -a~_0D^{1-\alpha_1}_tG_1(0)-~_0D^{1-\alpha_1}_tA G_1(0)+a~_0D^{1-\alpha_2}_tG_2(0) \quad\quad\quad\,\, {\rm in}\ \Omega,\ t\in[0,T],\\
			&\frac{\partial V_2}{\partial t}+a~_0D^{1-\alpha_2}_tV_2+~_0D^{1-\alpha_2}_tA V_2-a~_0D^{1-\alpha_1}_tV_1=\\
			&\qquad\qquad -a~_0D^{1-\alpha_2}_tG_2(0)-~_0D^{1-\alpha_2}_tA G_2(0)+a~_0D^{1-\alpha_1}_tG_1(0)  \quad\quad\,\,\quad {\rm in}\ \Omega,\ t\in[0,T],\\
			&\mathbf{V}(\cdot,0)=0 ~~~~~~~~~~~~~~~\quad\quad\quad\quad\quad\quad\quad\quad\quad\quad\quad\quad\quad\quad\quad\quad\quad\quad\quad\quad\quad\quad\quad\quad\quad\quad {\rm in}\ \Omega,\\
			&\mathbf{V}=0 ~~\quad\quad\quad\quad\quad\quad\quad\quad\quad\quad\quad\quad\quad\quad\quad\quad\quad\quad\quad\quad\quad\quad\quad\quad\quad\quad\quad\quad {\rm on}\ \partial\Omega,\ t\in[0,T].
		\end{aligned}
		\right .
	\end{equation}
	Therefore we get the solutions of system \eqref{equO2retosol} as
	\begin{equation}\label{equdefVii}
		\begin{aligned}
			\tilde{V}_1=&H_{\alpha_1}(z)z^{\alpha_1-1}(-az^{-\alpha_1}G_1(0)-z^{-\alpha_1}AG_1(0))+H_{\alpha_1}(z)z^{\alpha_1-1}(az^{-\alpha_2}G_2(0))\\
			&+aH(z)z^{\alpha_1-1}(-az^{-\alpha_2}G_2(0)-z^{-\alpha_2}AG_2(0))+aH(z)z^{\alpha_1-1}(az^{-\alpha_1}G_1(0))\\
			=&\tilde{V}_{1,1}+\tilde{V}_{1,2}+\tilde{V}_{1,3}+\tilde{V}_{1,4},\\
			\tilde{V}_2=&aH(z)z^{\alpha_2-1}(-az^{-\alpha_1}G_1(0)-z^{-\alpha_1}AG_1(0))+aH(z)z^{\alpha_2-1}(az^{-\alpha_2}G_2(0))\\
			&+H_{\alpha_2}(z)z^{\alpha_2-1}(-az^{-\alpha_2}G_2(0)-z^{-\alpha_2}AG_2(0))+H_{\alpha_2}(z)z^{\alpha_2-1}(az^{-\alpha_1}G_1(0))\\
			=&\tilde{V}_{2,1}+\tilde{V}_{2,2}+\tilde{V}_{2,3}+\tilde{V}_{2,4},\\
		\end{aligned}
	\end{equation}
where `$\tilde{~}$' stands for taking Laplace transform.
Let $\bar{V}^{-1}_1=\bar{V}^{-1}_2=0$, and then \eqref{equO2tradis} can be rewritten as
	\begin{equation}\label{equO2tradis2}
		\left \{\begin{aligned}
			&\frac{1}{\tau}\left(\frac{3}{2}\bar{V}^n_1-2\bar{V}^{n-1}_1+\frac{1}{2}\bar{V}^{n-2}_1\right)+a\sum_{i=0}^{n-1}d^{1-\alpha_1}_{2,i}\bar{V}^{n-i}_{1}
			+\sum_{i=0}^{n-1}d^{1-\alpha_1}_{2,i}A \bar{V}^{n-i}_{1}-a\sum_{i=0}^{n-1}d^{1-\alpha_2}_{2,i}\bar{V}^{n-i}_{2}\\
			&\quad=-(a+A)\left (\sum_{i=0}^{n-1}d^{1-\alpha_1}_{2,i}+\frac{1}{2}d^{1-\alpha_1}_{2,n-1}\right )G_{1}(0)+a\left (\sum_{i=0}^{n-1}d^{1-\alpha_2}_{2,i}+\frac{1}{2}d^{1-\alpha_2}_{2,n-1}\right )G_{2}(0)\quad {\rm in}~\Omega,~n\geq 1,\\
			&\frac{1}{\tau}\left(\frac{3}{2}\bar{V}^n_2-2\bar{V}^{n-1}_2+\frac{1}{2}\bar{V}^{n-2}_2\right)+a\sum_{i=0}^{n-1}d^{1-\alpha_2}_{2,i}\bar{V}^{n-i}_{2}+\sum_{i=0}^{n-1}d^{1-\alpha_2}_{2,i}A \bar{V}^{n-i}_{2}-a\sum_{i=0}^{n-1}d^{1-\alpha_1}_{2,i}\bar{V}^{n-i}_{1}\\
			&\quad =-(a+A)\left (\sum_{i=0}^{n-1}d^{1-\alpha_2}_{2,i}+\frac{1}{2}d^{1-\alpha_2}_{2,n-1}\right )G_{2}(0)+a\left (\sum_{i=0}^{n-1}d^{1-\alpha_1}_{2,i}+\frac{1}{2}d^{1-\alpha_1}_{2,n-1}\right )G_{1}(0)\quad {\rm in}~\Omega,~n\geq 1,\\
			&\bar{V}^0_{1}=\bar{V}^0_{2}=0\qquad\qquad\qquad\qquad\qquad\qquad\qquad\qquad\qquad\qquad\qquad\qquad\qquad\qquad\qquad {\rm in}~\Omega,\\
			&\bar{V}^n_{1}=\bar{V}^n_{2}=0\qquad\qquad\qquad\qquad\qquad\qquad\qquad\qquad\qquad\qquad\qquad\qquad\qquad\qquad\quad\ \ {\rm on}~\partial\Omega,~n\geq 0.\\
		\end{aligned}\right .
	\end{equation}
	Multiplying $\zeta^n$ and summing from $1$ to $\infty$ for both sides of the first two formulas in \eqref{equO2tradis2}, we obtain
	\begin{equation*}
		\left .\begin{aligned}
			&\sum_{n=1}^{\infty}\frac{1}{\tau}\left(\frac{3}{2}\bar{V}^n_1-2\bar{V}^{n-1}_1+\frac{1}{2}\bar{V}^{n-2}_1\right)\zeta^n+a\sum_{n=1}^{\infty}\sum_{i=0}^{n-1}d^{1-\alpha_1}_{2,i}\bar{V}^{n-i}_{1}\zeta^n
			\\
			&+\sum_{n=1}^{\infty}\sum_{i=0}^{n-1}d^{1-\alpha_1}_{2,i}A \bar{V}^{n-i}_{1}\zeta^n-a\sum_{n=1}^{\infty}\sum_{i=0}^{n-1}d^{1-\alpha_2}_{2,i}\bar{V}^{n-i}_{2}\zeta^n=\\
			&\qquad-(a+A)\sum_{n=1}^{\infty}\left (\sum_{i=0}^{n-1}d^{1-\alpha_1}_{2,i}+\frac{1}{2}d^{1-\alpha_1}_{2,n-1}\right )G_{1}(0)\zeta^n+a\sum_{n=1}^{\infty}\left (\sum_{i=0}^{n-1}d^{1-\alpha_2}_{2,i}+\frac{1}{2}d^{1-\alpha_2}_{2,n-1}\right )G_{2}(0)\zeta^n,\\
			&\sum_{n=1}^{\infty}\frac{1}{\tau}\left(\frac{3}{2}\bar{V}^n_2-2\bar{V}^{n-1}_2+\frac{1}{2}\bar{V}^{n-2}_2\right)\zeta^n+a\sum_{n=1}^{\infty}\sum_{i=0}^{n-1}d^{1-\alpha_2}_{2,i}\bar{V}^{n-i}_{2}\zeta^n\\
			&+\sum_{n=1}^{\infty}\sum_{i=0}^{n-1}d^{1-\alpha_2}_{2,i}A \bar{V}^{n-i}_{2}\zeta^n- a\sum_{n=1}^{\infty}\sum_{i=0}^{n-1}d^{1-\alpha_1}_{2,i}\bar{V}^{n-i}_{1}\zeta^n=\\
			&\qquad-(a+A)\sum_{n=1}^{\infty}\left (\sum_{i=0}^{n-1}d^{1-\alpha_2}_{2,i}+\frac{1}{2}d^{1-\alpha_2}_{2,n-1}\right )G_{2}(0)\zeta^n+a\sum_{n=1}^{\infty}\left (\sum_{i=0}^{n-1}d^{1-\alpha_1}_{2,i}+\frac{1}{2}d^{1-\alpha_1}_{2,n-1}\right )G_{1}(0)\zeta^n.
		\end{aligned}\right .
	\end{equation*}
	Using Eq. \eqref{equdefofdj2} and the facts
	\begin{equation*}
		\begin{aligned}
			\zeta\left (\frac{3}{2}+\sum\limits_{n=1}^{\infty}\zeta^n\right )=&\left (\frac{3}{2}-\frac{\zeta}{2}\right )\sum\limits_{n=1}^{\infty}\zeta^n\\
			=&\left (\frac{3-\zeta}{2(1-\zeta)}\right )\zeta=:\nu(\zeta)
		\end{aligned}
	\end{equation*}
	and $\delta(\zeta)=\frac{(1-\zeta)+(1-\zeta)^2/2}{\tau}$, we obtain
	\begin{equation*}
		\begin{aligned}
			&\delta(\zeta)\sum_{n=1}^{\infty}\bar{V}^{n}_1\zeta^n+a\left (\delta(\zeta)\right )^{1-\alpha_1}\sum_{n=1}^{\infty}\bar{V}^{n}_{1}\zeta^n+\left (\delta(\zeta)\right )^{1-\alpha_1}\sum_{n=1}^{\infty}A \bar{V}^{n}_{1}\zeta^n\\
			&\qquad-a\left (\delta(\zeta)\right )^{1-\alpha_2}\sum_{n=1}^{\infty}\bar{V}^{n}_{2}\zeta^n=-(a+A)\left (\delta(\zeta)\right )^{1-\alpha_1}\nu(\zeta)G_{1}(0)+a\left (\delta(\zeta)\right )^{1-\alpha_2}\nu(\zeta)G_{2}(0),\\
			&\delta(\zeta)\sum_{n=1}^{\infty}\bar{V}^{n}_2\zeta^n+a\left (\delta(\zeta)\right )^{1-\alpha_2}\sum_{n=1}^{\infty}\bar{V}^{n}_{2}\zeta^n+\left (\delta(\zeta)\right )^{1-\alpha_2}\sum_{n=1}^{\infty}A \bar{V}^{n}_{2}\zeta^n\\
			&\qquad-a\left (\delta(\zeta)\right )^{1-\alpha_1}\sum_{n=1}^{\infty}\bar{V}^{n}_{1}\zeta^n=-(a+A)\left (\delta(\zeta)\right )^{1-\alpha_2}\nu(\zeta)G_{2}(0)+a\left (\delta(\zeta)\right )^{1-\alpha_1}\nu(\zeta)G_{1}(0).
		\end{aligned}
	\end{equation*}
	Thus, after simple calculations, we have
	\begin{equation}\label{equdefbarVii}
		\begin{aligned}
			\sum_{n=1}^{\infty}\bar{V}^{n}_1\zeta^n=&H_{\alpha_1}(\delta(\zeta))(\delta(\zeta))^{\alpha_1-1}(-(a+A)(\delta(\zeta))^{1-\alpha_1}\nu(\zeta)G_1(0))\\
			&+H_{\alpha_1}(\delta(\zeta))(\delta(\zeta))^{\alpha_1-1}a(\delta(\zeta))^{1-\alpha_2}\nu(\zeta)G_2(0)\\
			&+aH(\delta(\zeta))(\delta(\zeta))^{\alpha_1-1}(-(a+A)(\delta(\zeta))^{1-\alpha_2}\nu(\zeta)G_2(0))\\
			&+aH(\delta(\zeta))(\delta(\zeta))^{\alpha_1-1}a(\delta(\zeta))^{1-\alpha_1}\nu(\zeta)G_1(0)\\
			=&\sum_{n=1}^{\infty}\bar{V}^{n}_{1,1}\zeta^n+\sum_{n=1}^{\infty}\bar{V}^{n}_{1,2}\zeta^n+\sum_{n=1}^{\infty}\bar{V}^{n}_{1,3}\zeta^n+\sum_{n=1}^{\infty}\bar{V}^{n}_{1,4}\zeta^n,\\
			\sum_{n=1}^{\infty}\bar{V}^{n}_2\zeta^n=&aH(\delta(\zeta))(\delta(\zeta))^{\alpha_2-1}(-(a+A)(\delta(\zeta))^{1-\alpha_1}\nu(\zeta)G_1(0))\\
			&+aH(\delta(\zeta))(\delta(\zeta))^{\alpha_2-1}a(\delta(\zeta))^{1-\alpha_2}\nu(\zeta)G_2(0)\\
			&+H_{\alpha_2}(\delta(\zeta))(\delta(\zeta))^{\alpha_2-1}(-(a+A)(\delta(\zeta))^{1-\alpha_2}\nu(\zeta)G_2(0))\\
			&+H_{\alpha_2}(\delta(\zeta))(\delta(\zeta))^{\alpha_2-1}a(\delta(\zeta))^{1-\alpha_1}\nu(\zeta)G_1(0)\\
			=&\sum_{n=1}^{\infty}\bar{V}^{n}_{2,1}\zeta^n+\sum_{n=1}^{\infty}\bar{V}^{n}_{2,2}\zeta^n+\sum_{n=1}^{\infty}\bar{V}^{n}_{2,3}\zeta^n+\sum_{n=1}^{\infty}\bar{V}^{n}_{2,4}\zeta^n.
		\end{aligned}
	\end{equation}
	To get error estimate between $V_1(t_n)$ and $\bar{V}^n_1$, we need to get the error estimates between $V_{1,i}(t_n)$ and $\bar{V}^n_{1,i}$~$(i=1,2,3,4)$. Then we consider the error estimate between $V_{1,1}(t_n)$ and $\bar{V}^n_{1,1}$. Using Eq. \eqref{equdefbarVii} and denoting $\mu(\zeta)$ as $\mu(\zeta)=\tau\delta(\zeta)\nu(\zeta)=\frac{\zeta(3-\zeta)^2}{4}$, when $\xi_\tau=e^{-(\kappa+1)\tau}$,  we have
	\begin{equation*}
		\bar{V}^n_{1,1}=\frac{1}{2\pi\tau \mathbf{i}}\int_{|\zeta|=\xi_\tau}\zeta^{-n-1}H_{\alpha_1}(\delta(\zeta))(\delta(\zeta))^{\alpha_1-1}(-(a+A))(\delta(\zeta))^{-\alpha_1}\mu(\zeta)d\zeta G_1(0).
	\end{equation*}
	Taking $\zeta=e^{-z\tau}$, we obtain
	\begin{equation*}
		\bar{V}^n_{1,1}=\frac{1}{2\pi \mathbf{i}}\int_{\Gamma^\tau}e^{zt_n}H_{\alpha_1}(\delta(e^{-z\tau}))(\delta(e^{-z\tau}))^{\alpha_1-1}(-(a+A))(\delta(e^{-z\tau}))^{-\alpha_1}\mu(e^{-z\tau})dz G_1(0),
	\end{equation*}
	where $\Gamma^\tau=\{z=\kappa+1+\mathbf{i}y:y\in\mathbb{R}~{\rm and}~|y|\leq \pi/\tau\}$. Next we deform the contour $\Gamma^\tau$ to
	$\Gamma^\tau_{\theta,\kappa}=\{z\in \mathbb{C}:\kappa\leq |z|\leq\frac{\pi}{\tau\sin(\theta)},|\arg z|=\theta\}\bigcup\{z\in \mathbb{C}:|z|=\kappa,|\arg z|\leq\theta\}$. Thus
	\begin{equation}\label{equdefV11}
		\bar{V}^n_{1,1}=\frac{1}{2\pi \mathbf{i}}\int_{\Gamma^\tau_{\theta,\kappa}}e^{zt_n}H_{\alpha_1}(\delta(e^{-z\tau}))(\delta(e^{-z\tau}))^{\alpha_1-1}(-(a+A))(\delta(e^{-z\tau}))^{-\alpha_1}\mu(e^{-z\tau})dz G_1(0).
	\end{equation}
	Taking the inverse Laplace transform for \eqref{equdefVii}, and combining \eqref{equdefV11}, we obtain
	\begin{equation*}
		\begin{aligned}
			&V_{1,1}(t_n)-\bar{V}^{n}_{1,1}\\
			=&\frac{1}{2\pi \mathbf{i}}\int_{\Gamma_{\theta,\kappa}}e^{zt_n}H_{\alpha_1}(z)z^{\alpha_1-1}(-(a+A))z^{-\alpha_1}dz G_1(0)\\
			&-\frac{1}{2\pi \mathbf{i}}\int_{\Gamma^\tau_{\theta,\kappa}}e^{zt_n}H_{\alpha_1}(\delta(e^{-z\tau}))(\delta(e^{-z\tau}))^{\alpha_1-1}(-(a+A))(\delta(e^{-z\tau}))^{-\alpha_1}\mu(e^{-z\tau})dz G_1(0)\\
			=&\frac{1}{2\pi \mathbf{i}}\int_{\Gamma_{\theta,\kappa}\backslash\Gamma^\tau_{\theta,\kappa}}e^{zt_n}H_{\alpha_1}(z)z^{-1}(-(a+A))dz G_1(0)\\
			&+\frac{1}{2\pi \mathbf{i}}\int_{\Gamma^\tau_{\theta,\kappa}}e^{zt_n}\left (H_{\alpha_1}(z)z^{-1}(-(a+A))-H_{\alpha_1}(\delta(e^{-z\tau}))(\delta(e^{-z\tau}))^{-1}(-(a+A))\right )dz G_1(0)\\
			&+\frac{1}{2\pi \mathbf{i}}\int_{\Gamma^\tau_{\theta,\kappa}}e^{zt_n}H_{\alpha_1}(\delta(e^{-z\tau}))(\delta(e^{-z\tau}))^{-1}(-(a+A))(1-\mu(e^{-z\tau}))dz G_1(0)\\
			=&\uppercase\expandafter{\romannumeral1}+\uppercase\expandafter{\romannumeral2}+\uppercase\expandafter{\romannumeral3}.
		\end{aligned}
	\end{equation*}
	According to Lemma \ref{lemaAH}, we have
	\begin{equation*}
		\|\uppercase\expandafter{\romannumeral1}\|_{L^2(\Omega)}\leq C\int_{\Gamma_{\theta,\kappa}\backslash\Gamma^{\tau}_{\theta,\kappa}}e^{-C|z|t_n}|z|^{-1}|dz|\|G_1(0)\|_{L^2(\Omega)}\leq C\tau^2 t_n^{-2}\|G_1(0)\|_{L^2(\Omega)}.
	\end{equation*}
	Using $\|\frac{d}{dz}H_{\alpha_1}(z)z^{-1}(-(a+A))\|\leq C|z|^{-2}$, $\delta(e^{-z\tau})=z+\mathcal{O}(\tau^2z^3)$ and the mean value theorem, we obtain
	\begin{equation*}
		\|\uppercase\expandafter{\romannumeral2}\|_{L^2(\Omega)}\leq C\tau^2\int_{\Gamma^\tau_{\theta,\kappa}}e^{\Re(z)t_n}|z||dz|\|G_1(0)\|_{L^2(\Omega)}\leq C\tau^2 t_n^{-2}\|G_1(0)\|_{L^2(\Omega)}.
	\end{equation*}
Combining the fact $\mu(e^{-z\tau})=1+O(z^2\tau^2)$ \cite{Lubich1996,Jin2017} and Lemmas \ref{lemaAH} and \ref{lemdelta} leads to
	\begin{equation*}
		\|\uppercase\expandafter{\romannumeral3}\|_{L^2(\Omega)}\leq C\tau^2\int_{\Gamma^\tau_{\theta,\kappa}}e^{\Re(z)t_n}|z||dz|\|G_1(0)\|_{L^2(\Omega)}\leq C\tau^2 t_n^{-2}\|G_1(0)\|_{L^2(\Omega)}.
	\end{equation*}
	Thus
	\begin{equation*}
		\|V_{1,1}(t_n)-\bar{V}^n_{1,1}\|_{L^2(\Omega)}\leq C\tau^2 t_n^{-2}\|G_1(0)\|_{L^2(\Omega)}.
	\end{equation*}
	Similarly, we can get
	\begin{equation*}
		\begin{aligned}
			&\|V_{1,2}(t_n)-\bar{V}^n_{1,2}\|_{L^2(\Omega)}\leq C\tau^2 t_n^{\alpha_2-2}\|G_2(0)\|_{L^2(\Omega)},\\
			&\|V_{1,3}(t_n)-\bar{V}^n_{1,3}\|_{L^2(\Omega)}\leq C\tau^2 t_n^{\alpha_2-2}\|G_2(0)\|_{L^2(\Omega)},\\
			&\|V_{1,4}(t_n)-\bar{V}^n_{1,4}\|_{L^2(\Omega)}\leq C\tau^2 t_n^{\alpha_1+\alpha_2-2}\|G_1(0)\|_{L^2(\Omega)}.\\
		\end{aligned}
	\end{equation*}
	Using the fact $T/t>1$, we have
	\begin{equation*}
		\|V_1(t_n)-\bar{V}^n_1\|_{L^2(\Omega)}\leq C\tau^2\left( t_n^{-2}\|G_1(0)\|_{L^2(\Omega)}+ t_n^{\alpha_2-2}\|G_2(0)\|_{L^2(\Omega)}\right).
	\end{equation*}
	Similarly, there is
	\begin{equation*}
		\|V_2(t_n)-\bar{V}^n_2\|_{L^2(\Omega)}\leq C\tau^2\left( t_n^{\alpha_1-2}\|G_1(0)\|_{L^2(\Omega)}+ t_n^{-2}\|G_2(0)\|_{L^2(\Omega)}\right).
	\end{equation*}
	Thus, the proof is completed.
\end{proof}
Then we have the following regularity estimates of the solutions of system \eqref{equO2tradis}.
\begin{theorem}\label{thmO2trareg}
	Let $\bar{G}^n_{1}$, $\bar{G}^n_{2}$ be the solutions of the systems \eqref{equO2tradis}. Then
	\begin{equation*}
		\begin{aligned} 
			&\| \bar{G}^n_1\|_{L^2(\Omega)}\leq C\left (\|G_{1}(0)\|_{L^2(\Omega)}+\|G_{2}(0)\|_{L^2(\Omega)}\right ),\\
			&\| \bar{G}^n_2\|_{L^2(\Omega)}\leq C\left (\|G_{1}(0)\|_{L^2(\Omega)}+\|G_{2}(0)\|_{L^2(\Omega)}\right ).
		\end{aligned}
	\end{equation*}
\end{theorem}
\begin{proof}
	The proof is similar to the proof of Theorem \ref{thmO1trareg}.
\end{proof}
Next we provide the error estimate between $\bar{G}^n_1$, $\bar{G}^n_2$ and $G^n_1$, $G^n_2$, which are the solutions of the systems \eqref{equO2tradis} and \eqref{equO2fastdis}, respectively.
\begin{theorem}\label{thmO2errortrafast}
	Let $\bar{G}^n_{1}$, $\bar{G}^n_{2}$ and $G^n_{1}$, $G^n_{2}$  be, respectively, the solutions of the systems \eqref{equO2tradis} and \eqref{equO2fastdis}. Then
	\begin{equation*}
		\begin{aligned}
			\|G^n_{1}-\bar{G}^n_{1}\|_{L^2(\Omega)}+
			\|G^n_{2}-\bar{G}^n_{2}\|_{L^2(\Omega)} \leq Cn\max_{0\leq i\leq n}(\epsilon_{i})(\|G_1(0)\|_{L^2(\Omega)}+\|G_2(0)\|_{L^2(\Omega)}),
		\end{aligned}
	\end{equation*}
	where $\epsilon_{i}=\max(|\epsilon^{1-\alpha_{1}}_{2,i}|,|\epsilon^{1-\alpha_{2}}_{2,i}|)$.
\end{theorem}
\begin{proof}
	In the proof, we take $\kappa\geq 1/t$ for given $t$ and ensure $\kappa$ is large enough to satisfy the conditions in Lemmas \ref{lemtheestimateofHblabla} and  \ref{lemaAH}. Subtracting \eqref{equO2fastdis} from \eqref{equO2tradis}, denoting $e^n_{1}=\bar{G}^n_1-G^n_1$, $e^n_{2}=\bar{G}^n_2-G^n_2$, $e^{-1}_1=e^{-1}_2=0$ and using $e^0_1=e^0_2=0$, we have
	\begin{equation*}
		\begin{aligned}
			&\frac{1}{\tau}\left(\frac{3}{2}e^n_1-2e^{n-1}_1+\frac{1}{2}e^{n-2}_1\right)+a\sum_{i=0}^{n-1}d^{1-\alpha_1}_{2,i}e^{n-i}_{1}
			+\sum_{i=0}^{n-1}d^{1-\alpha_1}_{2,i}A e^{n-i}_{1}\\
			&\qquad-a\sum_{i=0}^{n-1}d^{1-\alpha_2}_{2,i}e^{n-i}_{2}=-\sum_{i=0}^{n-1}\epsilon^{1-\alpha_1}_{2,i}(aG^{n-i}_1+AG^{n-i}_1)+\sum_{i=0}^{n-1}\epsilon^{1-\alpha_2}_{2,i}aG^{n-i}_2,\\
			&\frac{1}{\tau}\left(\frac{3}{2}e^n_2-2e^{n-1}_2+\frac{1}{2}e^{n-2}_2\right)+a\sum_{i=0}^{n-1}d^{1-\alpha_2}_{2,i}e^{n-i}_{2}+\sum_{i=0}^{n-1}d^{1-\alpha_2}_{2,i}A e^{n-i}_{2}\\
			&\qquad- a\sum_{i=0}^{n-1}d^{1-\alpha_1}_{2,i}e^{n-i}_{1}=-\sum_{i=0}^{n-1}\epsilon^{1-\alpha_2}_{2,i}(aG^{n-i}_2+AG^{n-i}_2)+\sum_{i=0}^{n-1}\epsilon^{1-\alpha_1}_{2,i}aG^{n-i}_1.\\
		\end{aligned}
	\end{equation*}
	By simple calculation, we obtain
	\begin{equation}\label{equerrorsbd}
		\begin{aligned}
			&\frac{1}{\tau}\left(\frac{3}{2}e^n_1-2e^{n-1}_1+\frac{1}{2}e^{n-2}_1\right)+a\sum_{i=0}^{n-1}d^{1-\alpha_1}_{2,i}e^{n-i}_{1}
			+\sum_{i=0}^{n-1}d^{1-\alpha_1}_{2,i}A e^{n-i}_{1}\\
			&\qquad-a\sum_{i=0}^{n-1}d^{1-\alpha_2}_{2,i}e^{n-i}_{2}=\sum_{i=0}^{n-1}\epsilon^{1-\alpha_1}_{2,i}(ae^{n-i}_1+Ae^{n-i}_1)-\sum_{i=0}^{n-1}\epsilon^{1-\alpha_2}_{2,i}ae^{n-i}_2\\
			&\qquad\qquad-\sum_{i=0}^{n-1}\epsilon^{1-\alpha_1}_{2,i}(a\bar{G}^{n-i}_1+A\bar{G}^{n-i}_1)+\sum_{i=0}^{n-1}\epsilon^{1-\alpha_2}_{2,i}a\bar{G}^{n-i}_2,\\
			&\frac{1}{\tau}\left(\frac{3}{2}e^n_2-2e^{n-1}_2+\frac{1}{2}e^{n-2}_2\right)+a\sum_{i=0}^{n-1}d^{1-\alpha_2}_{2,i}e^{n-i}_{2}+\sum_{i=0}^{n-1}d^{1-\alpha_2}_{2,i}A e^{n-i}_{2}\\
			&\qquad- a\sum_{i=0}^{n-1}d^{1-\alpha_1}_{2,i}e^{n-i}_{1}=\sum_{i=0}^{n-1}\epsilon^{1-\alpha_2}_{2,i}(ae^{n-i}_2+Ae^{n-i}_2)-\sum_{i=0}^{n-1}\epsilon^{1-\alpha_1}_{2,i}ae^{n-i}_1\\
			&\qquad\qquad-\sum_{i=0}^{n-1}\epsilon^{1-\alpha_2}_{2,i}(a\bar{G}^{n-i}_2+A\bar{G}^{n-i}_2)+\sum_{i=0}^{n-1}\epsilon^{1-\alpha_1}_{2,i}a\bar{G}^{n-i}_1.
		\end{aligned}
	\end{equation}
	Multiplying $\zeta^n$ and summing from $1$ to $\infty$ for the both sides of\ equations in \eqref{equerrorsbd} lead to
	\begin{equation*}
		\begin{aligned}
			&\sum_{n=1}^{\infty}\frac{1}{\tau}\left(\frac{3}{2}e^n_1-2e^{n-1}_1+\frac{1}{2}e^{n-2}_1\right)\zeta^n+a\sum_{n=1}^{\infty}\sum_{i=0}^{n-1}d^{1-\alpha_1}_{2,i}e^{n-i}_{1}\zeta^n
			+\sum_{n=1}^{\infty}\sum_{i=0}^{n-1}d^{1-\alpha_1}_{2,i}A e^{n-i}_{1}\zeta^n\\
			&\qquad-a\sum_{n=1}^{\infty}\sum_{i=0}^{n-1}d^{1-\alpha_2}_{2,i}e^{n-i}_{2}\zeta^n=\sum_{n=1}^{\infty}\sum_{i=0}^{n-1}\epsilon^{1-\alpha_1}_{2,i}(ae^{n-i}_1+Ae^{n-i}_1)\zeta^n-\sum_{n=1}^{\infty}\sum_{i=0}^{n-1}\epsilon^{1-\alpha_2}_{2,i}ae^{n-i}_2\zeta^n\\
			&\qquad\qquad-\sum_{n=1}^{\infty}\sum_{i=0}^{n-1}\epsilon^{1-\alpha_1}_{2,i}(a\bar{G}^{n-i}_1+A\bar{G}^{n-i}_1)\zeta^n+\sum_{n=1}^{\infty}\sum_{i=0}^{n-1}\epsilon^{1-\alpha_2}_{2,i}a\bar{G}^{n-i}_2\zeta^n,\\
			&\sum_{n=1}^{\infty}\frac{1}{\tau}\left(\frac{3}{2}e^n_2-2e^{n-1}_2+\frac{1}{2}e^{n-2}_2\right)\zeta^n+a\sum_{n=1}^{\infty}\sum_{i=0}^{n-1}d^{1-\alpha_2}_{2,i}e^{n-i}_{2}\zeta^n+\sum_{n=1}^{\infty}\sum_{i=0}^{n-1}d^{1-\alpha_2}_{2,i}A e^{n-i}_{2}\zeta^n\\
			&\qquad- a\sum_{n=1}^{\infty}\sum_{i=0}^{n-1}d^{1-\alpha_1}_{2,i}e^{n-i}_{1}\zeta^n=\sum_{n=1}^{\infty}\sum_{i=0}^{n-1}\epsilon^{1-\alpha_2}_{2,i}(ae^{n-i}_2+Ae^{n-i}_2)\zeta^n-\sum_{n=1}^{\infty}\sum_{i=0}^{n-1}\epsilon^{1-\alpha_1}_{2,i}ae^{n-i}_1\zeta^n\\
			&\qquad\qquad-\sum_{n=1}^{\infty}\sum_{i=0}^{n-1}\epsilon^{1-\alpha_2}_{2,i}(a\bar{G}^{n-i}_2+A\bar{G}^{n-i}_2)\zeta^n+\sum_{n=1}^{\infty}\sum_{i=0}^{n-1}\epsilon^{1-\alpha_1}_{2,i}a\bar{G}^{n-i}_1\zeta^n.
		\end{aligned}
	\end{equation*}
	Introducing
	\begin{equation*}
		\varepsilon^{\alpha}_2(\zeta)=\sum_{i=0}^{\infty}\epsilon^{\alpha}_{2,i}\zeta^i
	\end{equation*}
	and using Eq. \eqref{equdefofdj2}, we obtain
	\begin{equation*}
		\begin{aligned}
			&\delta(\zeta)\sum_{n=1}^{\infty}e_1^n\zeta^n+a(\delta(\zeta))^{1-\alpha_1}\sum_{n=1}^{\infty}e^{n}_{1}\zeta^n
			+(\delta(\zeta))^{1-\alpha_1}\sum_{n=1}^{\infty}A e^{n}_{1}\zeta^n\\
			&\qquad-a(\delta(\zeta))^{1-\alpha_2}\sum_{n=1}^{\infty}e^{n}_{2}\zeta^n=\varepsilon^{1-\alpha_1}_{2}(\zeta)\sum_{n=1}^{\infty}(ae^{n}_1+Ae^{n}_1)\zeta^n-\varepsilon^{1-\alpha_2}_{2}(\zeta)\sum_{n=1}^{\infty}ae^{n}_2\zeta^n\\
			&\qquad\qquad-\varepsilon^{1-\alpha_1}_{2}(\zeta)\sum_{n=1}^{\infty}(a\bar{G}^{n}_1+A\bar{G}^{n}_1)\zeta^n+\varepsilon^{1-\alpha_2}_{2}(\zeta)\sum_{n=1}^{\infty}a\bar{G}^{n}_2\zeta^n,\\
			&\delta(\zeta)\sum_{n=1}^{\infty}e_2^n\zeta^n+a(\delta(\zeta))^{1-\alpha_2}\sum_{n=1}^{\infty}e^{n}_{2}\zeta^n+(\delta(\zeta))^{1-\alpha_2}\sum_{n=1}^{\infty}A e^{n}_{2}\zeta^n\\
			&\qquad- a(\delta(\zeta))^{1-\alpha_1}\sum_{n=1}^{\infty}e^{n}_{1}\zeta^n=\varepsilon^{1-\alpha_2}_{2}(\zeta)\sum_{n=1}^{\infty}(ae^{n}_2+Ae^{n}_2)\zeta^n-\varepsilon^{1-\alpha_1}_{2}(\zeta)\sum_{n=1}^{\infty}ae^{n}_1\zeta^n\\
			&\qquad\qquad-\varepsilon^{1-\alpha_2}_{2}(\zeta)\sum_{n=1}^{\infty}(a\bar{G}^{n}_2+A\bar{G}^{n}_2)\zeta^n+\varepsilon^{1-\alpha_1}_{2}(\zeta)\sum_{n=1}^{\infty}a\bar{G}^{n}_1\zeta^n.
		\end{aligned}
	\end{equation*}
	Thus
	\begin{equation*}
		\begin{aligned}
			\sum_{n=1}^{\infty}e^n_1\zeta^n=&H_{\alpha_1}(\delta(\zeta))(\delta(\zeta))^{\alpha_1-1}\left(\varepsilon^{1-\alpha_1}_{2}(\zeta)\sum_{n=1}^{\infty}(ae^{n}_1+Ae^{n}_1)\zeta^n-\varepsilon^{1-\alpha_2}_{2}(\zeta)\sum_{n=1}^{\infty}ae^{n}_2\zeta^n\right.\\
			&\left.-\varepsilon^{1-\alpha_1}_{2}(\zeta)\sum_{n=1}^{\infty}(a\bar{G}^{n}_1+A\bar{G}^{n}_1)\zeta^n+\varepsilon^{1-\alpha_2}_{2}(\zeta)\sum_{n=1}^{\infty}a\bar{G}^{n}_2\zeta^n\right)\\
			&+aH(\delta(\zeta))(\delta(\zeta))^{\alpha_1-1}\left(\varepsilon^{1-\alpha_2}_{2}(\zeta)\sum_{n=1}^{\infty}(ae^{n}_2+Ae^{n}_2)\zeta^n-\varepsilon^{1-\alpha_1}_{2}(\zeta)\sum_{n=1}^{\infty}ae^{n}_1\zeta^n\right.\\
			&\left.-\varepsilon^{1-\alpha_2}_{2}(\zeta)\sum_{n=1}^{\infty}(a\bar{G}^{n}_2+A\bar{G}^{n}_2)\zeta^n+\varepsilon^{1-\alpha_1}_{2}(\zeta)\sum_{n=1}^{\infty}a\bar{G}^{n}_1\zeta^n\right),\\
			\sum_{n=1}^{\infty}e^n_2\zeta^n=&aH(\delta(\zeta))(\delta(\zeta))^{\alpha_2-1}\left(\varepsilon^{1-\alpha_1}_{2}(\zeta)\sum_{n=1}^{\infty}(ae^{n}_1+Ae^{n}_1)\zeta^n-\varepsilon^{1-\alpha_2}_{2}(\zeta)\sum_{n=1}^{\infty}ae^{n}_2\zeta^n\right.\\
			&\left.-\varepsilon^{1-\alpha_1}_{2}(\zeta)\sum_{n=1}^{\infty}(a\bar{G}^{n}_1+A\bar{G}^{n}_1)\zeta^n+\varepsilon^{1-\alpha_2}_{2}(\zeta)\sum_{n=1}^{\infty}a\bar{G}^{n}_2\zeta^n\right)\\
			&+H_{\alpha_2}(\delta(\zeta))(\delta(\zeta))^{\alpha_2-1}\left(\varepsilon^{1-\alpha_2}_{2}(\zeta)\sum_{n=1}^{\infty}(ae^{n}_2+Ae^{n}_2)\zeta^n-\varepsilon^{1-\alpha_1}_{2}(\zeta)\sum_{n=1}^{\infty}ae^{n}_1\zeta^n\right.\\
			&\left.-\varepsilon^{1-\alpha_2}_{2}(\zeta)\sum_{n=1}^{\infty}(a\bar{G}^{n}_2+A\bar{G}^{n}_2)\zeta^n+\varepsilon^{1-\alpha_1}_{2}(\zeta)\sum_{n=1}^{\infty}a\bar{G}^{n}_1\zeta^n\right).
		\end{aligned}
	\end{equation*}
	Denote
	\begin{equation*}
		\begin{aligned}
			&\sum_{i=0}^{\infty}E^i_{0,1,1}\zeta^i=a^2H\left (\delta(\zeta)\right )\left (\delta(\zeta)\right )^{\alpha_1-1}\varepsilon^{1-\alpha_1}_2(\zeta),\\
			&\sum_{i=0}^{\infty}E^i_{0,1,2}\zeta^i=aH\left (\delta(\zeta)\right )\left (\delta(\zeta)\right )^{\alpha_1-1}(a+A)\varepsilon^{1-\alpha_2}_2(\zeta),\\
			&\sum_{i=0}^{\infty}E^i_{0,2,1}\zeta^i=aH\left (\delta(\zeta)\right )\left (\delta(\zeta)\right )^{\alpha_2-1}(a+A)\varepsilon^{1-\alpha_1}_2(\zeta),\\
			&\sum_{i=0}^{\infty}E^i_{0,2,2}\zeta^i=a^2H\left (\delta(\zeta)\right )\left (\delta(\zeta)\right )^{\alpha_2-1}\varepsilon^{1-\alpha_2}_2(\zeta),\\
			&\sum_{i=0}^{\infty}E^i_{1,1,1}\zeta^i=H_{\alpha_1}\left (\delta(\zeta)\right )\left (\delta(\zeta)\right )^{\alpha_1-1}(a+A)\varepsilon^{1-\alpha_1}_2(\zeta),\\
			&\sum_{i=0}^{\infty}E^i_{1,1,2}\zeta^i=aH_{\alpha_1}\left (\delta(\zeta)\right )\left (\delta(\zeta)\right )^{\alpha_1-1}\varepsilon^{1-\alpha_2}_2(\zeta),\\
			&\sum_{i=0}^{\infty}E^i_{1,2,1}\zeta^i=aH_{\alpha_2}\left (\delta(\zeta)\right )\left (\delta(\zeta)\right )^{\alpha_2-1}\varepsilon^{1-\alpha_1}_2(\zeta),\\
			&\sum_{i=0}^{\infty}E^i_{1,2,2}\zeta^i=H_{\alpha_2}\left (\delta(\zeta)\right )\left (\delta(\zeta)\right )^{\alpha_2-1}(a+A)\varepsilon^{1-\alpha_2}_2(\zeta).
		\end{aligned}
	\end{equation*}
	Here we first consider the estimate of $E^n_{0,1,1}$.
	Taking $\xi_{\tau}=e^{-\tau(\kappa+1)}$, we get
	\begin{equation*}
		E^n_{0,1,1}=\frac{1}{2\pi \mathbf{i}}\int_{|\zeta|=\zeta_{\tau}}\zeta^{-n-1}a^2H\left (\delta(\zeta)\right )\left (\delta(\zeta)\right )^{\alpha_1-1}\varepsilon^{1-\alpha_1}_2(\zeta)d\zeta,
	\end{equation*}
	which leads to
	\begin{equation*}
		\|E^n_{0,1,1}\|\leq C\tau^{-1}\max_{0\leq i\leq n}(|\epsilon^{1-\alpha_1}_{2,i}|)\int_{|\zeta|=\zeta_{\tau}}|\zeta|^{-n-1}\left \|a^2H\left (\delta(\zeta)\right )\right \| \left |\delta(\zeta)\right |^{\alpha_1-1}\left|\frac{1-\zeta}{\tau}\right|^{-1}|d\zeta|.
	\end{equation*}
Letting $\zeta=e^{-z\tau}$, there is
	\begin{equation*}
		\|E^n_{0,1,1}\|\leq C\max_{0\leq i\leq n}(|\epsilon^{1-\alpha_1}_{2,i}|)\int_{\Gamma^{\tau}}|e^{zt_n}|\left \|a^2H\left (\delta(e^{-z\tau})\right )\right \| \left |\delta(e^{-z\tau})\right |^{\alpha_1-1}\left|\frac{1-\zeta}{\tau}\right|^{-1}|dz|,
	\end{equation*}
	where $\Gamma^\tau=\{z=\kappa+1+\mathbf{i}y:y\in\mathbb{R}~{\rm and}~|y|\leq \pi/\tau\}$. Next,  we deform the contour $\Gamma^\tau$ to
	$\Gamma^\tau_{\theta,\kappa}=\{z\in \mathbb{C}:\kappa\leq |z|\leq\frac{\pi}{\tau\sin(\theta)},|\arg z|=\theta\}\bigcup\{z\in \mathbb{C}:|z|=\kappa,|\arg z|\leq\theta\}$. Thus
	\begin{equation*}
		\|E^n_{0,1,1}\|\leq C\max_{0\leq i\leq n}(|\epsilon^{1-\alpha_1}_{2,i}|)\int_{\Gamma^{\tau}_{\theta,\kappa}}|e^{zt_n}|\left \|a^2H\left (\delta(e^{-z\tau})\right )\right \| \left |\delta(e^{-z\tau})\right |^{\alpha_1-1}\left|\frac{1-\zeta}{\tau}\right|^{-1}|dz|.
	\end{equation*}
	Using Lemmas \ref{lemtheestimateofHblabla} and \ref{lemdelta}, we have
	\begin{equation*}
		\|E^n_{0,1,1}\|\leq C\max_{0\leq i\leq n}(|\epsilon^{1-\alpha_1}_{2,i}|)\int_{\Gamma^{\tau}_{\theta,\kappa}}e^{\Re(z)t_n}|z|^{-\alpha_2-2}|dz|\leq C\max_{0\leq i\leq n}(\epsilon_{i}).
	\end{equation*}
	Similarly, we have
	\begin{equation*}
		\|E^n_{0,1,2}\|,\|E^n_{0,2,1}\|,\|E^n_{0,2,2}\|,\|E^n_{1,1,1}\|,\|E^n_{1,1,2}\|,\|E^n_{1,2,1}\|,\|E^n_{1,2,2}\|\leq C\max_{0\leq i\leq n}(\epsilon_{i}).
	\end{equation*}
	Thus
	\begin{equation*}
		\begin{aligned}
			\|e^n_1\|_{L^2(\Omega)}+\|e^n_2\|_{L^2(\Omega)}\leq& C\max_{0\leq i\leq n}(\epsilon_{i})\sum_{i=1}^{n-1}(\|e^i_1\|_{L^2(\Omega)}		+\|e^i_2\|_{L^2(\Omega)})\\
			&+C\max_{0\leq i\leq n}(\epsilon_{i})\sum_{i=1}^{n}(\|\bar{G}^i_1\|_{L^2(\Omega)}+\|\bar{G}^i_2\|_{L^2(\Omega)}).
		\end{aligned}
	\end{equation*}
Combining Gr\"{o}nwall inequality and Theorem \ref{thmO2trareg} leads to
	\begin{equation*}
		\|e^n_1\|_{L^2(\Omega)}+\|e^n_2\|_{L^2(\Omega)}\leq Cn\max_{0\leq i\leq n}(\epsilon_{i})(\|G_1(0)\|_{L^2(\Omega)}+\|G_2(0)\|_{L^2(\Omega)}).
	\end{equation*}	
\end{proof}
Combining Theorems \ref{thmO2CQerr} and \ref{thmO2errortrafast}, we get the following error estimates for the fast SBD scheme.
\begin{theorem}
	Let $G_{1}$, $G_{2}$ and $G^n_{1}$, $G^n_{2}$ be the solutions of the systems \eqref{equrqtosol} and \eqref{equO2fastdis}, respectively. Then we have the estimates, if $G_{1}(0)$, $G_{2}(0)\in L^2(\Omega)$,
	\begin{equation*}
		\begin{aligned}
			\|G_{1}(t_n)-G^n_{1}\|_{L^2(\Omega)}\leq&C\tau^2\left(t_n^{-2}\|G_{1}(0)\|_{L^2(\Omega)}+t_n^{\alpha_2-2}\|G_{2}(0)\|_{L^2(\Omega)}\right)\\
			&+Cn\max_{0\leq i\leq n}(\epsilon_{i})\left(\|G_1(0)\|_{L^2(\Omega)}+\|G_2(0)\|_{L^2(\Omega)}\right),\\
			\|G_{2}(t_n)-G^n_{2}\|_{L^2(\Omega)}\leq& C\tau^2\left(t_n^{\alpha_1-2}\|G_{1}(0)\|_{L^2(\Omega)}+t_n^{-2}\|G_{2}(0)\|_{L^2(\Omega)}\right)\\
			&+Cn\max_{0\leq i\leq n}(\epsilon_{i})\left(\|G_1(0)\|_{L^2(\Omega)}+\|G_2(0)\|_{L^2(\Omega)}\right),
		\end{aligned}
	\end{equation*}
	where $\epsilon_{i}=\max(|\epsilon^{1-\alpha_{1}}_{2,i}|,|\epsilon^{1-\alpha_{2}}_{2,i}|)$.
\end{theorem}

\section{Numerical experiments}\label{sec:5}
In this section, we verify the effectiveness of the fast algorithms by comparing with the classical convolution quadrature schemes. Here, we denote $h$ as the mesh size and define
\begin{equation*}
	\begin{aligned}
		E_{i,\tau}=\|G^N_{i,\tau}-G_{i}(t_N)\|_{L^2(\Omega)}, ~i=1,2,
	\end{aligned}
\end{equation*}
where $G^N_{1,\tau}$ and $G^N_{2,\tau}$, respectively, signify the numerical solutions of $G_{1}$ and $G_{2}$ at the fixed time $t_N$ with time step size $\tau$. The temporal convergence rates can be calculated by
\begin{equation*}
	 {\rm Rate}=\frac{\ln(E_{i,\tau}/E_{i,\tau/2})}{\ln(2)},\  i=1,2.
\end{equation*}

In the numerical experiments, the following two groups of initial values are used:
\begin{enumerate}[(a)]
	\item\label{initial_a} \begin{equation*}
		G_{1}(x,0)=x(1-x),\qquad G_{2}(x,0)=\sin(x);
	\end{equation*}
	\item\label{initial_b} \begin{equation*}
	G_{1}(x,0)=(1-x)\chi_{(1/2,1)},\qquad G_{2}(x,0)=x\chi_{(0,1/2)},
	\end{equation*}
\end{enumerate}
where $\chi_{(a,b)}$ denotes the characteristic function on $(a,b)$.
\subsection{Performance of fast BE scheme}
We,  respectively, use the BE scheme and fast BE scheme to solve the system \eqref{equrqtosol} with $a=2$. Use the numerical solution with $\tau=1/3200$ and $h=1/256$ as the `exact' solution. Tables \ref{tab:O1smooth} and \ref{tab:O1nonsmooth} give the $L_2$ errors, convergence rates and CPU time for solving the system \eqref{equrqtosol} with the initial values  \eqref{initial_a} and \eqref{initial_b} for different $\alpha_1$ and $\alpha_2$ respectively, which show that the fast BE scheme has the same convergence rates as BE scheme and it takes much less CPU time when $\tau$ is small.
\begin{table}
\footnotesize\center
	\caption{$L_2$ error, convergence rates, and CPU time at $t=1$ with $h=1/256$}
	\label{tab:O1smooth}
\begin{tabular}{c|c|c|ccccc}
\hline	
$(\alpha_1,\alpha_2)$	&            &$1/\tau$            &        100 &        200 &        400 &        800 &       1600 \\
\hline	
	&            &  $E_{1,\tau}$          &  8.434E-06 &  4.141E-06 &  2.001E-06 &  9.335E-07 &  4.000E-07 \\
	
	&            &            &       Rate &    1.0263  &    1.0489  &    1.1003  &    1.2228  \\
	
	&         BE &     $E_{2,\tau}$       &  1.347E-04 &  6.609E-05 &  3.193E-05 &  1.489E-05 &  6.379E-06 \\
	
	&            &            &       Rate &    1.0276  &    1.0496  &    1.1007  &    1.2230  \\
	
	(0.3,0.6)&            &   CPU time(s)         &      1.34  &      3.19  &      7.08  &     22.33  &     78.53  \\
	\cline{2-8}
	 &            &   $E_{1,\tau}$         &  8.434E-06 &  4.141E-06 &  2.001E-06 &  9.335E-07 &  4.000E-07 \\
	
	&            &            &       Rate &    1.0263  &    1.0489  &    1.1003  &    1.2228  \\
	
	&        FBE &      $E_{2,\tau}$      &  1.347E-04 &  6.609E-05 &  3.193E-05 &  1.489E-05 &  6.379E-06 \\
	
	&            &            &       Rate &    1.0277  &    1.0496  &    1.1007  &    1.2229  \\
	
	&            &  CPU time(s)          &      1.39  &      2.92  &      5.78  &     11.97  &     23.95  \\
	\hline
	&            &  $E_{1,\tau}$          &  9.318E-06 &  4.574E-06 &  2.211E-06 &  1.031E-06 &  4.417E-07 \\
	
	&            &            &       Rate &    1.0266  &    1.0490  &    1.1004  &    1.2228  \\
	
	&         BE &    $E_{2,\tau}$        &  1.390E-04 &  6.811E-05 &  3.289E-05 &  1.533E-05 &  6.568E-06 \\
	
	&            &            &       Rate &    1.0290  &    1.0503  &    1.1010  &    1.2231  \\
	
	(0.4,0.7) &            &      CPU time(s)      &      1.48  &      3.14  &      7.27  &     22.42  &     79.66  \\
	\cline{2-8}
	&            &   $E_{1,\tau}$         &  9.318E-06 &  4.574E-06 &  2.210E-06 &  1.031E-06 &  4.417E-07 \\
	
	&            &            &       Rate &    1.0266  &    1.0490  &    1.1004  &    1.2228  \\
	
	&        FBE &   $E_{2,\tau}$         &  1.390E-04 &  6.811E-05 &  3.289E-05 &  1.533E-05 &  6.568E-06 \\
	
	&            &            &       Rate &    1.0290  &    1.0503  &    1.1010  &    1.2231  \\
	
	&            &  CPU time(s)          &      1.47  &      2.98  &      5.72  &     12.33  &     24.61  \\
\hline	
\end{tabular}
\end{table}

\begin{table}
\footnotesize\center
	\caption{$L_2$ error, convergence rates, and CPU time at $t=1$ with $h=1/256$}
	\label{tab:O1nonsmooth}
	\begin{tabular}{c|c|c|ccccc}
		\hline
	$(\alpha_1,\alpha_2)$	&            &    $1/\tau$        &        100 &        200 &        400 &        800 &       1600 \\
	\hline	
		&            &  $E_{1,\tau}$            &  1.916E-05 &  9.411E-06 &  4.550E-06 &  2.123E-06 &  9.095E-07 \\
		
		&            &            &       Rate &    1.0254  &    1.0485  &    1.1001  &    1.2227  \\
		
		&         BE &     $E_{2,\tau}$         &  3.467E-05 &  1.699E-05 &  8.204E-06 &  3.824E-06 &  1.638E-06 \\
		
		&            &            &       Rate &    1.0290  &    1.0502  &    1.1010  &    1.2231  \\
		
		(0.3,0.7)&            &  CPU time(s)          &      1.36  &      3.16  &      7.13  &     22.09  &     78.36  \\
		\cline{2-8}
		 &            &   $E_{1,\tau}$          &   1.916E-05 &  9.411E-06 &  4.550E-06 &  2.123E-06 &  9.095E-07 \\
		
		&            &            &       Rate &    1.0254  &    1.0485  &    1.1001  &    1.2227  \\
		
		&        FBE &    $E_{2,\tau}$          &  3.467E-05 &  1.699E-05 &  8.203E-06 &  3.824E-06 &  1.638E-06 \\
		
		&            &            &       Rate &    1.0290  &    1.0502  &    1.1010  &    1.2231  \\
		
		&            &   CPU time(s)         &      1.41  &      2.84  &      5.89  &     11.94  &     24.02  \\
		\hline
		&            &   $E_{1,\tau}$           &  2.510E-05 &  1.233E-05 &  5.959E-06 &  2.779E-06 &  1.191E-06 \\
		
		&            &            &       Rate &    1.0260  &    1.0487  &    1.1003  &    1.2228  \\
		
		&         BE &  $E_{2,\tau}$            &  3.364E-05 &  1.650E-05 &  7.971E-06 &  3.717E-06 &  1.592E-06 \\
		
		&            &            &       Rate &    1.0276  &    1.0496  &    1.1007  &    1.2230  \\
		
		(0.4,0.6) &            &   CPU time(s)         &      1.41  &      3.08  &      7.16  &     22.17  &     79.25  \\
		\cline{2-8}
		&            &    $E_{1,\tau}$          &  2.510E-05 &  1.233E-05 &  5.959E-06 &  2.779E-06 &  1.191E-06 \\
		
		&            &            &       Rate &    1.0260  &    1.0487  &    1.1002  &    1.2227  \\
		
		&        FBE &    $E_{2,\tau}$          &  3.364E-05 &  1.650E-05 &  7.971E-06 &  3.717E-06 &  1.592E-06 \\
		
		&            &            &       Rate &    1.0276  &    1.0496  &    1.1006  &    1.2229  \\
		
		&            &  CPU time(s)          &      1.41  &      2.92  &      6.03  &     12.05  &     24.41  \\
		\hline
	\end{tabular}  	
\end{table}

\subsection{Performance of fast SBD scheme}
Here we first discuss the choice of $N_s$. According to  \eqref{eqintegraltoapproO2}, we need to make $|\epsilon^\alpha_{2,i}|$ small enough such that the approximation of $d^{\alpha}_{2,i}$ effective, and there are two ways to achieve it:  the increase of the number of the integral points and the selection of a suitable parameter $N_s$.  For small $N_s$, one needs to use a large number of integral points to ensure the accuracy of the approximation of $d^\alpha_{2,i}$ because of the low convergence rates of the Gauss-Jacobi rule,  which increases the computation time and storage cost  tremendously.  Figure \ref{fig:031d362} shows the change of the error  $|\epsilon^\alpha_{2,i}|$ as $i$ increases when $T=1$, $\tau=1/1000$,  $\alpha=0.3$, and  $N_{p,1}+N_{p,2}=62$. It is found that when $i<15$, $d^{\alpha}_{2,i}$ can't be well approximated, i.e., $|\epsilon^\alpha_{2,i}|$ is large, so we start from the $15$-th term to approximate $d^\alpha_{2,i}$, that is, we take $N_s=15$ to ensure the accuracy of the approximation. The same thing happens in Figure \ref{fig:081d382} when $T=1$, $\tau=1/1000$, $\alpha=0.8$, and  $N_{p,1}+N_{p,2}=82$, hence we take $N_s=17$ in the fast SBD scheme. In a word,  choosing a suitable value for $N_s$ not only ensures the accuracy of the scheme, but also saves the computation time.

\begin{figure}
	\begin{center}
		\caption{$\alpha=0.3$, $\tau=1/1000$, $N_{p,1}+N_{p,2}=62$}\label{fig:031d362}
		\includegraphics[width=13.66cm,height=6cm,angle=0]{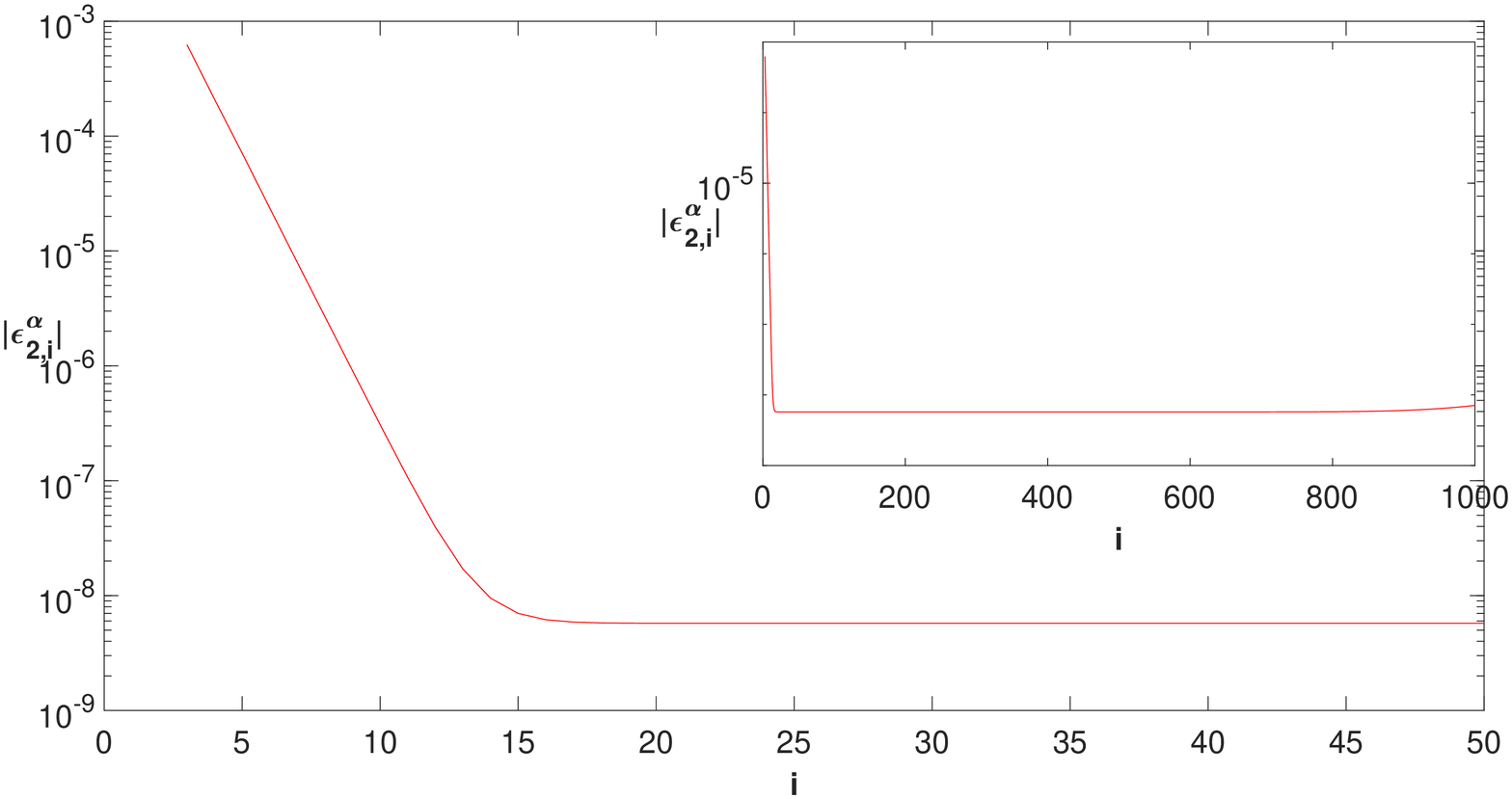}
	\end{center}
\end{figure}

\begin{figure}
	\begin{center}
		\caption{$\alpha=0.8$, $\tau=1/1000$, $N_{p,1}+N_{p,2}=82$}\label{fig:081d382}
		\includegraphics[width=13.66cm,height=6cm,angle=0]{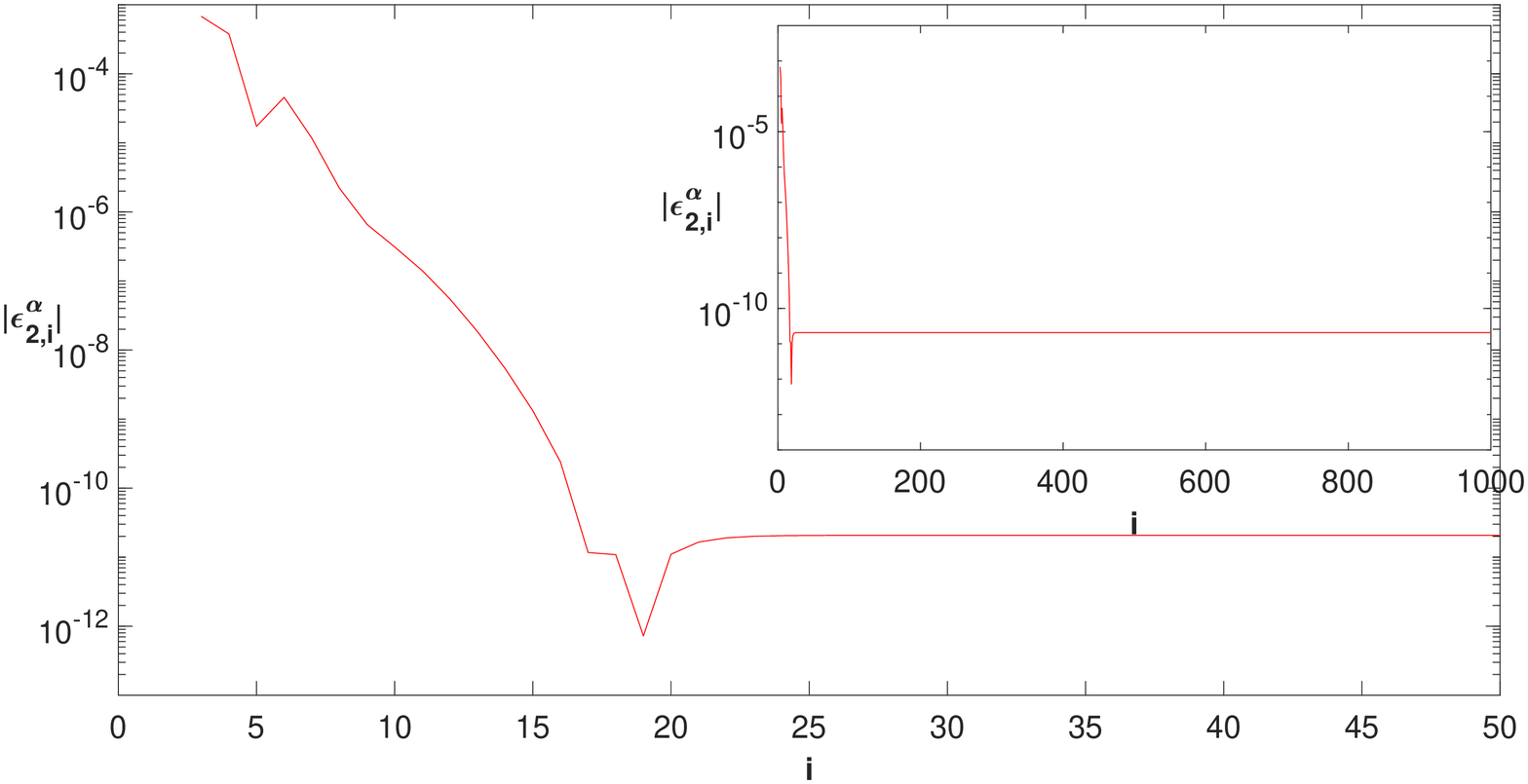}
	\end{center}
\end{figure}

Then we solve the system \eqref{equrqtosol} by the SBD scheme and fast SBD scheme, respectively, with $a=-1$. Use the numerical solution with $\tau=1/640$ and $h=1/1024$ as the `exact' solution. Tables \ref{tab:O2smooth} and \ref{tab:O2nonsmooth} provide the $L_2$ errors and convergence rates for solving the system \eqref{equrqtosol}, respectively,  with the initial values \eqref{initial_a} and \eqref{initial_b} for different $\alpha_1$ and $\alpha_2$, which show that the fast SBD scheme has the same convergence rates as SBD scheme.
\begin{table}
\footnotesize\center
	\caption{$L_2$ error  and convergence rates at $t=1$ with $h=1/1024$}
	\label{tab:O2smooth}
	\begin{tabular}{c|c|c|ccccc}
		\hline
		$(\alpha_1,\alpha_2)$&            &  $1/\tau$          &         20 &         40 &         80 &        160 &        320 \\
		\hline
		&            &    $E_{1,\tau}$        &  9.445E-06 &  2.274E-06 &  5.521E-07 &  1.304E-07 &  2.598E-08 \\
		
		&            &            &          Rate &    2.0546  &    2.0419  &    2.0825  &    2.3273  \\
		
		&         SBD &    $E_{2,\tau}$        &  3.976E-05 &  9.551E-06 &  2.317E-06 &  5.467E-07 &  1.090E-07 \\
		(0.3,0.4)
		&            &            &          Rate &    2.0577  &    2.0435  &    2.0831  &    2.3271  \\

		\cline{2-8}
		&            &   $E_{1,\tau}$         &  9.463E-06 &  2.274E-06 &  5.519E-07 &  1.303E-07 &  2.611E-08 \\
		
		&            &            &          Rate &    2.0572  &    2.0427  &    2.0829  &    2.3190  \\
		
		&        FSBD &    $E_{2,\tau}$        &  3.981E-05 &  9.552E-06 &  2.316E-06 &  5.465E-07 &  1.095E-07 \\
		
		&            &            &          Rate &    2.0593  &    2.0442  &    2.0831  &    2.3188  \\
		
		\hline
		&            &    $E_{1,\tau}$        &  3.189E-05 &  7.594E-06 &  1.835E-06 &  4.343E-07 &  8.803E-08 \\
		
		&            &            &          Rate &    2.0703  &    2.0487  &    2.0792  &    2.3028  \\
		
		&         SBD &   $E_{2,\tau}$         &  9.081E-05 &  2.158E-05 &  5.225E-06 &  1.255E-06 &  2.675E-07 \\
		
		(0.7,0.8)&            &            &          Rate &    2.0733  &    2.0461  &    2.0581  &    2.2296  \\

		\cline{2-8}
		&            &    $E_{1,\tau}$        &  3.193E-05 &  7.594E-06 &  1.836E-06 &  4.349E-07 &  8.947E-08 \\
		
		&            &            &          Rate &    2.0720  &    2.0487  &    2.0774  &    2.2812  \\
		
		&        FSBD &    $E_{2,\tau}$        &  9.089E-05 &  2.158E-05 &  5.228E-06 &  1.263E-06 &  2.863E-07 \\
		
		&            &            &          Rate &    2.0744  &    2.0454  &    2.0499  &    2.1406  \\
		
		\hline
	\end{tabular}
\end{table}

\begin{table}
\footnotesize\center
	\caption{$L_2$ error and convergence rates at $t=1$ with $h=1/1024$}
	\label{tab:O2nonsmooth}
	\begin{tabular}{c|c|c|ccccc}
		\hline
	$(\alpha_1,\alpha_2)$	&            &    $1/\tau$        &         20 &         40 &         80 &        160 &        320 \\
		\hline
		&            &  $E_{1,\tau}$         &  3.262E-06 &  7.860E-07 &  1.910E-07 &  4.511E-08 &  8.992E-09 \\
		
		&            &            &       Rate &    2.0534  &    2.0412  &    2.0817  &    2.3269  \\
		
		&        SBD &   $E_{2,\tau}$         &  5.178E-06 &  1.246E-06 &  3.024E-07 &  7.139E-08 &  1.422E-08 \\
		
			(0.2,0.4) &            &            &       Rate &    2.0555  &    2.0424  &    2.0827  &    2.3275  \\

		\cline{2-8}
		&            &    $E_{1,\tau}$        &  3.228E-06 &  7.859E-07 &  1.909E-07 &  4.507E-08 &  9.028E-09 \\
		
		&            &            &       Rate &    2.0382  &    2.0415  &    2.0825  &    2.3197  \\
		
		&        FSBD &     $E_{2,\tau}$       &  5.152E-06 &  1.246E-06 &  3.023E-07 &  7.134E-08 &  1.429E-08 \\
		
		&            &            &       Rate &    2.0482  &    2.0429  &    2.0832  &    2.3193  \\

		\hline
		&            &   $E_{1,\tau}$         &  1.141E-05 &  2.729E-06 &  6.607E-07 &  1.561E-07 &  3.136E-08 \\
		
		&            &            &       Rate &    2.0636  &    2.0461  &    2.0814  &    2.3155  \\
		
		&         SBD &  $E_{2,\tau}$          &  1.543E-05 &  3.669E-06 &  8.887E-07 &  2.137E-07 &  4.577E-08 \\
		
		(0.6,0.8)&            &            &       Rate &    2.0727  &    2.0455  &    2.0561  &    2.2231  \\
		
		\cline{2-8}
		&            &  $E_{1,\tau}$          &  1.142E-05 &  2.729E-06 &  6.606E-07 &  1.562E-07 &  3.160E-08 \\
		
		&            &            &       Rate &    2.0655  &    2.0464  &    2.0804  &    2.3053  \\
		
		&        FSBD &    $E_{2,\tau}$        &  1.545E-05 &  3.669E-06 &  8.892E-07 &  2.151E-07 &  4.918E-08 \\
		
		&            &           &       Rate &    2.0740  &    2.0448  &    2.0473  &    2.1290  \\

		\hline
	\end{tabular}
\end{table}

Finally, we use the SBD scheme and fast SBD scheme to solve the system \eqref{equrqtosol} with $a=-1$ to verify the efficiency of the latter scheme. Use the numerical solution with $\tau=1/4000$ and $h=1/1024$ as the `exact' solution. Table \ref{tab:O2time} shows the $L_2$ error and CPU time when $T=10$, $\tau=1/2000$ and $h=1/1024$ for different $\alpha_1$ and $\alpha_2$. The $L^2$ errors of SBD scheme are close to the errors of fast SBD scheme but the CPU times of  fast SBD scheme are much less than  SBD scheme for different $\alpha_1$ and $\alpha_2$, which show that the fast algorithm can greatly reduce computation time.

\begin{table}
\footnotesize\center
	\caption{$L^2$ error and CPU time at t=10}
	\label{tab:O2time}
	\begin{tabular}{c|c|ccc}
		\hline
		&          $(\alpha_1,\alpha_2)$  &  (0.3,0.8) &  (0.4,0.7) &  (0.5,0.6) \\
		\hline
		&    $E_{1,\tau}$        &  2.628E-10 &  4.019E-10 &  7.399E-10 \\
		
		SBD &  $E_{2,\tau}$          &  7.571E-09 &  3.471E-09 &  1.561E-09 \\
		
		& CPU time(s) &    870.13  &    883.70  &    882.48  \\
		\cline{2-5}
		&      $E_{1,\tau}$      &  3.543E-10 &  4.862E-10 &  1.008E-09 \\
		
		FSBD &     $E_{2,\tau}$       &  1.321E-08 &  5.697E-09 &  2.375E-09 \\
		
		& CPU time(s) &    209.44  &    210.83  &    211.36  \\
		\hline
	\end{tabular}
\end{table}

\section*{Conclusion}
The fast algorithms based on BE and SBD convolution quadratures are developed to solve the homogeneous fractional Fokker-Planck equations with two internal states, and they, respectively, have the first- and second-order convergence rates. One of the advantages of the provided fast algorithms is that the assumption of the regularity of the solution in time is not required. The effectiveness of the fast algorithms is verfied by numerical experiments.


\section*{Acknowledgements}
This work was supported by the National Natural Science Foundation of China under grant no. 11671182, and the Fundamental Research Funds for the Central Universities under grant no. lzujbky-2018-ot03.

\appendix
\section{Derivation of  \eqref{equO1toderinAp}}\label{AP1}
For $i\geq 2$, letting $\tau s=t$ in Eq. \eqref{equO1intrep}, we have
\begin{equation*}
	\begin{aligned}
		\int_0^{\infty}s^{\alpha-1} \frac{\tau s}{(\tau s+1)^{i+1}}ds=&\tau^{-\alpha}\int_0^{\infty}t^\alpha\left (\frac{1}{1+t}\right )^{i+1}dt\\
		=&\tau^{-\alpha}\int_1^{\infty}(t-1)^\alpha\left (\frac{1}{t}\right )^{i+1}dt.
	\end{aligned}
\end{equation*}
Then taking $\eta=1/t$ leads to
\begin{equation*}
	\begin{aligned}
		\tau^{-\alpha}\int_1^{\infty}(t-1)^\alpha\left (\frac{1}{t}\right )^{i+1}dt=&\tau^{-\alpha}\int_1^{\infty}\left(\frac{1}{\eta}-1\right) ^{\alpha} \eta^{i+1}d\frac{1}{\eta}\\
		=&\tau^{-\alpha}\int_0^{1}\left(\frac{1}{\eta}-1\right) ^{\alpha} \eta^{i-1}d\eta\\
		=&\tau^{-\alpha}\int_0^{1}\left(1-\eta\right) ^{\alpha} \eta^{-\alpha}\eta^{i-1}d\eta.\\
	\end{aligned}
\end{equation*}
Lastly, we take $\eta=(s+1)/2$ and get
\begin{equation*}
	\begin{aligned}
		\tau^{-\alpha}\int_0^{1}\left(1-\eta\right) ^{\alpha} \eta^{-\alpha}\eta^{i-1}d\eta=&\tau^{-\alpha}\int_0^{1}\left(1-\frac{s+1}{2}\right)^{\alpha} \left(\frac{s+1}{2}\right)^{i-1-\alpha}d\frac{s+1}{2}\\
		=&\tau^{-\alpha}2^{-2}\int_{-1}^{1}\left(1-s\right)^{\alpha}(1+s)^{1-\alpha} \left(\frac{s+1}{2}\right)^{i-2}ds.
	\end{aligned}
\end{equation*}

\section{Derivation of \eqref{equO2toderinAp}}\label{AP2}
For $i\geq 3$, taking $\tau s=t$ in \eqref{equO2intrep} results in
\begin{equation*}
	\begin{aligned}
		\int_0^{\infty}s^{\alpha-1}\frac{\tau s}{1+\sigma}\left(\frac{1}{3+\sigma}\right)^{i+1}ds=\tau^{-\alpha}\int_0^{\infty}t^{\alpha}\frac{1}{1+\sigma}\left(\frac{1}{3+\sigma}\right)^{i+1}dt
	\end{aligned}
\end{equation*}
 and $\sigma^2+2\sigma+2\tau s=\sigma^2+2\sigma+2t=0$, which lead to
\begin{equation*}
	\begin{aligned}
		\tau^{-\alpha}\int_0^{\infty}t^{\alpha}\frac{1}{1+\sigma}\left(\frac{1}{3+\sigma}\right)^{i+1}dt=&\tau^{-\alpha}\int_{0}^\infty\left(-\frac{\sigma^2+2\sigma}{2}\right)^{\alpha}\frac{1}{1+\sigma}\left(\frac{1}{3+\sigma}\right)^{i+1}d\left(-\frac{\sigma^2+2\sigma}{2}\right)\\
		=&-\tau^{-\alpha}\int_{\Gamma_1}\left(-\frac{\sigma^2+2\sigma}{2}\right)^{\alpha}\left(\frac{1}{3+\sigma}\right)^{i+1}d\sigma
	\end{aligned}	
\end{equation*}
with $\Gamma_1=\{z=-1+\sqrt{1-2t},~0\leq t\leq \infty\}$. It is easy to find that for any $n\geq 3$ and $-\pi<\theta_1,\theta_2<\pi$ when $0<\alpha<1$,
\begin{equation*}
	\begin{aligned}
		\lim_{r\rightarrow\infty}\left |\int_{\theta_1}^{\theta_2}\left(-\frac{r^2e^{2\mathbf{i}\theta}+2re^{\mathbf{i}\theta}}{2}\right)^{\alpha}\left(\frac{1}{3+re^{\mathbf{i}\theta}}\right)^{n+1}\mathbf{i}re^{\mathbf{i}\theta}d\theta\right |\leq \lim_{r\rightarrow\infty}Cr^{2\alpha-n}=0;
	\end{aligned}
\end{equation*}
and according to Jordan's Lemma, we have
\begin{equation*}
	\begin{aligned}
		-\tau^{-\alpha}\int_{\Gamma_1}\left(-\frac{\sigma^2+2\sigma}{2}\right)^{\alpha}\left(\frac{1}{3+\sigma}\right)^{i+1}d\sigma=&-\tau^{-\alpha}\int_{0}^{\infty}\left(-\frac{\sigma(\sigma+2)}{2}\right)^{\alpha}\left(\frac{1}{3+\sigma}\right)^{i+1}d\sigma\\
		=&-\tau^{-\alpha}\int_{2}^{\infty}\left(-\frac{\sigma(\sigma-2)}{2}\right)^{\alpha}\left(\frac{1}{1+\sigma}\right)^{i+1}d\sigma.
	\end{aligned}
\end{equation*}
Taking $\sigma=1/t$ in the above equation results in
\begin{equation*}
	\begin{aligned}
		-\tau^{-\alpha}\int_{2}^{\infty}\left(-\frac{\sigma(\sigma-2)}{2}\right)^{\alpha}\left(\frac{1}{1+\sigma}\right)^{i+1}d\sigma=&-\tau^{-\alpha}\int_{2}^{\infty}\left(-\frac{1/t(1/t-2)}{2}\right)^{\alpha}\left(\frac{t}{1+t}\right)^{i+1}d\frac{1}{t}\\
		=&-\tau^{-\alpha}(-2)^{-\alpha}\int_{0}^{\frac{1}{2}}\left(1-2t\right)^{\alpha}t^{-2\alpha-2}\left(\frac{t}{1+t}\right)^{i+1}dt.
	\end{aligned}
\end{equation*}
Further letting $t=(s+1)/4$ leads to
\begin{equation*}
	\begin{aligned}
		&-\tau^{-\alpha}(-2)^{-\alpha}\int_{0}^{\frac{1}{2}}\left(1-2t\right)^{\alpha}t^{-2\alpha-2}\left(\frac{t}{1+t}\right)^{i+1}dt\\=&-\tau^{-\alpha}(-2)^{-\alpha}\int_{0}^{\frac{1}{2}}\left(1-\frac{s+1}{2}\right)^{\alpha}\left(\frac{s+1}{4}\right)^{-2\alpha-2}\left(\frac{s+1}{s+5}\right)^{i+1}d\frac{s+1}{4}\\
		=&-\tau^{-\alpha}(-1)^{-\alpha}(2)^{2+2\alpha}\int_{-1}^{1}(1-s)^{\alpha}(1+s)^{2-2\alpha}\left(\frac{s+1}{s+5}\right)^{i-3}\left(\frac{1}{s+5}\right)^{4}ds.
	\end{aligned}
\end{equation*}
Similarly, for the second part of Eq. \eqref{equO2intrep}, we can rewrite it as
\begin{equation*}
	\begin{aligned}
		&-\int_0^{\infty}s^{\alpha-1}\frac{\tau s}{1+\sigma}\left(\frac{1}{1-\sigma}\right)^{i+1}ds=\\
		&-\tau^{-\alpha}(2)^{-\alpha-3}\int_{-1}^{1}(1-s)^{\alpha}(1+s)^{2-2\alpha}\left(\frac{s+1}{2}\right)^{i-3}\left(3s+1\right)^{\alpha}ds.
	\end{aligned}
\end{equation*}

\end{document}